\documentclass[english,a4paper,11pt]{article}
\usepackage[a4paper,hmargin=1.0in,vmargin=1.0in]{geometry}
\usepackage[round]{natbib}
\usepackage[dvipsnames]{xcolor}
\usepackage[linktocpage=true,pagebackref=true,bookmarks=true,bookmarksopen,bookmarksnumbered]{hyperref}
\usepackage{sectsty}
\usepackage[english]{babel}
\usepackage[utf8]{inputenc}
\usepackage{amsfonts}
\usepackage{amssymb}
\usepackage{amsmath}
\usepackage{stmaryrd}
\usepackage{comment}
\usepackage{setspace}
\usepackage{bbm}
\usepackage{mathtools}
\usepackage{enumitem}
\usepackage{amsthm}
\usepackage{algorithm}
\usepackage{algorithmic}
\usepackage{tikz}
\usepackage{comment}
\usepackage{thmtools}
\usepackage{thm-restate}
\usepackage{booktabs}
\usepackage{array}
\usepackage{cleveref}

\newcommand{\R}{\mathbb{R}}
\newcommand{\N}{\mathbb{N}}

\newcommand{\E}{\mathbb{E}}

\newcommand{\obj}{{\cal E}}

\newcommand{\Nc}{\mathcal{N}}
\newcommand{\opt}{{\sf OPT}}
\renewcommand{\P}{\mathbb{P}}
\newcommand{\eps}{\epsilon}

\newcommand{\sma}{{s}}
\newcommand{\lar}{{l}}

\newcommand{\cont}[1]{C_{#1}}

\newcommand{\Var}{\mathrm{Var}}
\newcommand{\Cov}{\mathrm{Cov}}
\newcommand{\nota}[1]{\hat{#1}}

\renewcommand{\overline}[1]{\left\llbracket#1\right\rrbracket}

\newcommand{\black}[1]{{{\color{black}{#1}}}}

\newtheoremstyle{DStheorem}
  {\topsep}
  {\topsep}
  {\itshape}
  {0pt}
  {\scshape}
  {.}
  { }
  {\thmname{#1}\thmnumber{ #2}\thmnote{ (#3)}}
\theoremstyle{DStheorem}
\newtheorem{theorem}{Theorem}[section]
\newtheorem{lemma}[theorem]{Lemma}
\newtheorem{claim}[theorem]{Claim}
\newtheorem{definition}[theorem]{Definition}

\let\oldproofname=\proofname
\renewcommand{\proofname}{\rm\sc{\oldproofname}}

\sectionfont{\large} \subsectionfont{\normalsize}
\allowdisplaybreaks
\makeindex
\newcommand{\changelocaltocdepth}[1]{%
  \addtocontents{toc}{\protect\setcounter{tocdepth}{#1}}%
  \setcounter{tocdepth}{#1}%
}

\newcounter{marouanecounter}

\newcounter{omarcounter}

\begin{document}

\begin{titlepage}

\title{Assortment Optimization with Visibility Constraints%
\thanks{An extended abstract of this paper appeared in proceedings of the 25th International Conference on Integer Programming and Combinatorial Optimization (IPCO 2024).}}

\author{%
\quad\quad Th\'eo Barr\'e\thanks{Department of Industrial Engineering and Operations Research,
University of California
Berkeley, CA, USA. Email: {\tt theo\_barre@berkeley.edu}.}  \and
Omar El Housni\thanks{School of Operations Research and Information Engineering, Cornell Tech, Cornell University, NY, USA. Email: {\tt \{oe46,mi262,al748\}@cornell.edu}.}
\and
Marouane Ibn Brahim\footnotemark[3] \quad\quad\quad
\and
Andrea Lodi\footnotemark[3]
\and
Danny Segev\thanks{School of Mathematical Sciences and Coller School of Management, Tel Aviv University, Tel Aviv 69978, Israel. Email: {\tt segevdanny@tauex.tau.ac.il}.}  }

\date{}
\maketitle

\setcounter{page}{200}
\thispagestyle{empty}

\begin{abstract}
Motivated by  applications in e-retail and online advertising, we study  the problem of assortment optimization under visibility constraints \eqref{APV}. Here, we are given a universe of substitutable products and a stream of customers. The objective is to determine the optimal assortment of products to offer to each customer in order to maximize the total expected revenue, subject to exogenously-given {\em visibility constraints}, stating that each product should be shown to a minimum number of customers. We assume that customer choices follow a Multinomial Logit model (MNL). 

We provide a structural characterization of optimal assortments  and present \black{a linear} time algorithm for solving \ref{APV}. To this end, we introduce a novel function called the ``expanded revenue" of an assortment and establish its supermodularity; our algorithm takes advantage of this structural property. Additionally, we prove that \ref{APV}  can be formulated as a compact linear program. Next, we consider \ref{APV} with cardinality constraints, \black{which limit the maximum number of products that can be included in an assortment}. We prove \black{this problem} to be strongly NP-hard and not admitting a Fully Polynomial Time Approximation Scheme (FPTAS), even when all products have identical prices. Subsequently, we devise a Polynomial Time Approximation Scheme (PTAS) for \ref{APV} under cardinality constraints with identical prices. 
We also examine the revenue loss resulting from the enforcement of visibility constraints, comparing it to the unconstrained problem. To offset this loss, we propose a novel strategy to distribute the loss incurred among the products subject to visibility constraints, charging each vendor an amount proportional to their product's contribution to the revenue loss.
\end{abstract}


\end{titlepage}


\pagestyle{plain}
\setcounter{page}{1}

\section{Introduction}\label{sec:intro}
Assortment optimization is a crucial aspect of  decision making in many industries  such as  e-retail and online advertising. In this domain, our goal is to select a subset of available products  to offer to customers in order to maximize a context-appropriate objective function, 
such as revenue, profit, or market share. For example, e-retailers seek to select which products should be displayed to customers in order to maximize their expected revenue. Online  advertisers strategically select the most effective combination of advertisements to maximize user engagement and desired outcomes, such as click-through rates. 
Consequently, the choice of a well formed assortment is crucial due to inherent substitution effects, where a product's attractiveness depends not only on its intrinsic value but also on the concurrent alternatives presented at that time. For instance, offering a high-quality, high-priced product alongside a comparable product at a significantly lower price may result in diminishing sales for the higher-priced product, leading in turn to an unsatisfactory platform revenue. For a comprehensive overview of assortment optimization and its applications, we refer avid readers to 
related surveys and books \citep{kok2015assortment, Gallego2019RevenueMA}.

Traditionally, assortment optimization frameworks often overlook a crucial element in contemporary e-commerce, which is the intended exposure of specific products to customers, which we refer to as product {\em visibility}. In today's complex business landscape, where companies adhere to Service-Level Agreements (SLAs) with suppliers and prioritize sponsored promotions, product visibility within our assortment of choice is pivotal. 
SLAs  often define conditions for product representation, ensuring equitable visibility for each supplier's products on the platform. Moreover, the concept of sponsored products has gained traction in recent years, with brands willing to pay for prominent display and increased visibility. 
While these strategies influence consumer behavior in various ways,  solely focusing on products visibility without considering the broader context of assortment optimization can lead to an imbalanced product mix, resulting in reduced customer satisfaction and overall revenue.

In this paper, we introduce the notion of {\em visibility constraints} into the classical framework of assortment optimization. In essence, our purpose is to enforce a minimum level of display, \black{asking each product to be shown} at least a certain number of times in the displayed assortments. Such constraints model both Service-Level Agreements and sponsored products. In addition, they can capture settings where the platform wishes to instill some fairness notion among vendors by ensuring that each product is given a ``fair'' chance, in the sense of being shown to at least a certain number of customers. Deferring the formal model formulation to Section \ref{sec:form}, we now provide a brief informal introduction to our model. Specifically, we are given a universe of substitutable products along with a stream of $T$ customers. For each  customer, we determine an individual assortment from the universe of products. Subsequently, the customer either decides to purchase one of these products, or chooses to leave without purchasing any product. We assume that these decisions are governed by the well-known Multinomial Logit (MNL) choice model. That is, what differentiates our work from classical models is enforcing the constraint that each  product has to be shown a minimum number of times among the $T$ offered assortments noting that this minimum display requirement is given exogenously.   Our objective is to maximize the combined expected revenue collected across the stream of customers. We refer to this newly-introduced problem as {\em assortment optimization with visibility} constraints, concisely denoted as \ref{APV}.

\black{Alternatively to our model, one could also consider a randomized approach, where the goal is to determine a distribution over assortments rather than an exact collection, as explored in the concurrent work of Lu et al. (2024). While both models have merit, focusing on a stream of customers rather than a distribution over assortments offers key advantages. First, many advertising platforms, such as Google and Facebook , employ cost-per-mille (CPM) pricing, where advertisers pay for a guaranteed level of visibility based on the number of product displays. Explicitly modeling a customer stream ensures that these visibility commitments are met, whereas a distribution-based approach does not inherently guarantee the required displays. As a result, a correction phase may be necessary at the end of the time horizon to satisfy visibility constraints. Second, in our setting, computing a deterministic solution is a more natural choice. Introducing randomness is often justified when a deterministic approach is either intractable or leads to unfavorable properties. However, our model turns out to exhibit favorable structural and algorithmic properties, as presented throughout the paper.}

From a computational standpoint, \black{it is only natural to wonder about} the complexity of including visibility constraints. In fact, without this requirement, the problem in question reduces to the classic MNL-driven assortment optimization problem, for which we know that the optimal offer set is price-ordered (\cite{talluri2004revenue}), and can therefore be computed in polynomial time. However, by enforcing visibility constraints, we necessarily have to include certain products along the sequence of displayed assortments, possibly cannibalizing the sales of more profitable products due to substitution effects.
Consequently, determining an optimal sequence of assortments  appears to be a challenging question. 
Subsequently, a practically motivated setting concerns the introduction of cardinality constraints on the offered assortments\black{, which impose an upper bound on the number of products included in the assortments}. Such constraints have extensively been studied in earlier literature \citep{Sumida2020RevenueUtilityTI, gallego2014constrained, desir2020constrained, kunnumkal2023new, leitner2024exact}, as they allow us to capture multiple applications where vendors have limited shelf space, screen size, or purchasing budget.
Finally, a fundamental question is that of quantifying the revenue loss incurred by enforcing visibility constraints compared to their relaxed unconstrained counterpart. 
This challenge compels us to develop a pricing strategy that appropriately apportions the loss to different vendors based on the impact of their product on the overall revenue.
Such a scenario frequently occurs within the framework of SLAs, where typically, vendor-platform contracts include a clause that guarantees a certain level of visibility to the vendor's product. In return for this visibility, the vendor compensates the platform with an appropriate fee. \black{In general, advertising companies such as Google, Facebook or Amazon employ various pricing schemes. For instance, two of the most common are {\em cost-per-mille} (CPM) and {\em cost-per-clic} (CPC); see e.g., \cite{asdemir2012pricing}. In a nutshell, the former pricing model charges a fee for a fixed number of ad impressions, typically per thousand, while the latter model charges a fee per interaction with the ad, usually a click. Our work aligns with the CPM framework, as we enforce visibility constraints that require a minimum number of displays for each product.
}

\subsection{Main Contributions and Technical Ideas}\label{subsec:contributions}
    
In this paper, we introduce and study  assortment  optimization with visibility constraints \eqref{APV} under the Multinomial Logit choice model. Our first contribution resides in settling the complexity status of this problem, by developing \black{an exact} polynomial time algorithm. Next, we consider a natural extension, imposing a cardinality constraint on the offered assortment\black{, which limits its allowable number of products}. We show that this additional requirement makes the problem fundamentally harder, rendering it strongly NP-hard, even when all products have identical prices. This finding excludes the possibility of devising a fully polynomial time approximation Scheme (FPTAS), unless $\mathrm{P}=\mathrm{NP}$. On the positive side, we design a \black{randomized} polynomial time approximation scheme (PTAS) for the setting of equal prices.
Our subsequent goal is to quantify the revenue loss brought by visibility constraints, compared to the unconstrained assortment optimization problem, and to design a  strategy for  sharing this loss among the different vendors for which visibility constraints have been enforced.  Our main contributions can be briefly summarized as follows.

\begin{enumerate}
    \item {\bf Exact polynomial time algorithm}. In Section \ref{sec:apv}, our main technical contribution is to design an exact \black{linear} time algorithm for \ref{APV}. To this end, we introduce the notion of expanded revenue and expanded set of a given assortment, which will play a pivotal role in designing our algorithm. We leverage structural properties of the expanded revenue function to characterize the structure of optimal solutions, enabling us to develop an efficient algorithmic approach.
    
    \begin{enumerate}
        \item {\bf Expanded Revenue}. \black{We define the so-called expanded revenue function as follows:} Given a universe of products $\cal N$ and an assortment $A \subseteq {\cal N}$, the expanded revenue of $A$ is defined as the maximum revenue of any assortment that contains $A$. In turn, the expanded set of $A$ is the assortment that achieves the latter revenue. We show that the expanded revenue function is closely related to the objective function of \ref{APV} for a single customer, and provide a linear time algorithm to compute the expanded revenue.         

        \item {\bf Monotonicity and supermodularity}. We prove that the expanded revenue function possesses several useful properties. Namely, we establish a monotonicity property, showing that the expanded revenue of an assortment decreases with the addition of new products. Then, we prove the main theoretical result on which our final algorithm relies: supermodularity\footnote{
        \black{A function $f\colon \Omega\rightarrow \R$ is supermodular if $f(B\cup\{i\})-f(B) \geq f(A\cup\{i\}) - f(A)$ for all $A\subseteq B\subseteq \Omega$ and $i\in \Omega\setminus B$, .}} of the expanded revenue function.
        
        \item {\bf Combinatorial algorithm and LP formulation}. Building on the above-mentioned properties, we finally identify a simple nested structure for  optimal \ref{APV} solutions, and devise a polynomial time algorithm to efficiently compute such assortments. Additionally, we demonstrate that \ref{APV} can be formulated as a compact  linear program, thereby proposing an alternative algorithmic approach.
    \end{enumerate}

\item {\bf \ref{APV} with cardinality constraints.} In section \ref{sec:APVC}, we consider the practically-motivated extension of \ref{APV}, where there is an upper bound on the number of products that can be displayed in each assortment.
\begin{enumerate}
    \item {\bf Hardness.} First, we prove that \ref{APV}  with cardinality constraints is strongly NP-hard, even with equal prices, by linking its resolution to the $3$-\texttt{PARTITION} problem. Moreover, we extend our proof to rule out the existence of an FPTAS, unless $\mathrm{P}=\mathrm{NP}$.
    
    \item {\bf Polynomial-time approximation scheme (PTAS).} Our cornerstone algorithmic result for the equal price setting is the design of a \black{randomized} polynomial-time approximation scheme (PTAS), allowing us to approach optimal revenues within any degree of precision. \black{Specifically, for any $\eps>0$, our approximation scheme outputs a random assortment whose performance guarantee is, in expectation, within $1-\eps$ of optimal}. Our PTAS relies on a linearization of the objective function through the efficient guessing of a carefully chosen set of parameters. Then, we formulate a linear relaxation, and leverage its fractional solution to compute a random approximate (integer) solution, using a dependent rounding scheme. The latter solution is subsequently shown to be, in expectation, near optimal for \ref{APV} with cardinality constraints. \black{Some of these ideas are inspired by the work of  \cite{segev2021efficient} on rounding fractional} solutions for the Santa Claus problem.
\end{enumerate}

\item {\bf Price of visibility.} Clearly, the introduction of visibility constraints may possibly reduce the \black{optimal} expected revenue compared to the unconstrained assortment optimization problem. In Section \ref{sec:price}, we aim to quantify this revenue loss and to propose a fair strategy for distributing it among vendors based on their respective contributions to the overall loss. We devise a pricing strategy that, for each product which negatively impacts the overall revenue, charges the vendor a fraction of the loss \black{that is} proportional to the ratio between the negative contribution of the product and the sum of the negative contributions of all products. We demonstrate that this strategy satisfies natural fairness properties and exhibits favorable computational tractability.
\end{enumerate}

\subsection{Related Literature}

Assortment optimization under the Multinomial Logit (MNL) model has been a well-established research direction in revenue management. Initially introduced by \cite{luce1959individual}, with subsequent works by \cite{mcfadden1973conditional} and \cite{RePEc:ecm:emetrp:v:52:y:1984:i:5:p:1219-40}, the MNL  model has gained a great deal of popularity for modeling customer choices due to its structural simplicity, predictive power, ease of estimation and computational tractability compared to more complex alternatives. It has been extensively utilized in various research works  such as those of \cite{mahajan2001stocking}, \cite{talluri2004revenue}, \cite{ Rusmevichientong2010DynamicAO}, \cite{Sumida2020RevenueUtilityTI}, \cite{gao2021assortment}, \cite{housni2023maximum} and \cite{el2023joint}, to mention a few. As demonstrated by \cite{talluri2004revenue}, under the MNL model, the optimal unconstrained assortment is revenue-ordered, namely, it contains all products whose prices exceed a certain threshold, simplifying the optimization problem by avoiding the consideration of exponentially-many subsets. Moreover, \cite{Gallego2011AGA} proposed a linear programming formulation for this problem. \cite{Rusmevichientong2010DynamicAO} seem to have been the first to consider cardinality constraints, proving this setting to be solvable in polynomial time. \cite{desir2022capacitated} studied more general capacity constraints, showing they are NP-hard to address in the general case. \cite{Sumida2020RevenueUtilityTI} studied totally unimodular constraint structures in this context, and showed that the resulting problem can be reformulated as a linear program. 
However, when considering mixtures of MNL models (MMNL), the assortment optimization problem becomes NP-hard even with two classes of customers \citep{rusmevichientong2014assortment}, and $O(n^{1-\eps})$-hard to approximate in general \citep{desir2022capacitated}.

To the best of our knowledge, our paper is the first to study assortment optimization under visibility constraints.  The topic of visibility in assortment planning has been largely overlooked thus far: \cite{chen2022fair} studied visibility under a fairness approach, trying to enforce similar visibility for products with similar characteristics, while \cite{wang2023advertising} studied a version in which they can increase the attractiveness of some products through an advertising budget. 
Very recently, in a concurrent work, \cite{lu2023simple} considered an assortment optimization problem under the MNL model, subject to  fairness constraints similar to the visibility constraints presented in our paper. However, in contrast to our model, where individual assortments are offered to a predetermined number of customers, their study investigates a {\em single-customer} version, where random assortments can be offered. The latter version appears to be not as complex. For instance, as mentioned in Section \ref{subsec:contributions}, we establish that our model with cardinality constraints is strongly NP-hard and does not admit an FPTAS, whereas their model admits an FPTAS under cardinality constraints. Hence, different optimization techniques and approaches appear to be needed.

Interestingly, the topic of assortment optimization for a stream of customers has typically been studied from an online perspective, where decisions are made sequentially; see, e.g., \cite{golrezaei2014real}
and \cite{aouad2023online}. \black{In contrast, we study a non-adaptive setting, where the entirety of our assortments is planned in advance, that is, the collection of assortments in decided offline}. Other existing versions, such as those studied by  \cite{Li2009ASA}, consider a flow of customers with randomized preferences, to which we offer a common assortment. Finally, in revenue management, pricing problems are often considered in the sense of optimizing the selling price of each product \citep{Wang2012CapacitatedAA,Miao2018DynamicJA, Alptekinolu2015TheEC}, while we focus on the scenario where selling prices are fixed,  and instead, we study the question of pricing the revenue loss generated by enforcing visibility of each product.

\section{Model Formulation}\label{sec:form}

{\bf The MNL choice model.} Let  $\mathcal{N} \coloneqq \{1,\ldots, n\}$ be a universe of substitutable products at our disposal, where each product $i \in {\cal N}$ has a  price $p_i \geq 0$.  Without loss of generality, we index these products by non-increasing prices, i.e., $p_1 \geq \cdots \geq p_n$.  
An assortment of products (or an offer set), is simply a subset of products $S \subseteq {\cal N}$. Additionally, the option of not selecting any product is symbolically represented as product $0$, referred to as the no-purchase option.

We assume that customers make their purchasing decisions according to a Multinomial Logit model. Under this model, each product $i \in {\cal N}$ is associated with a so-called preference weight $v_i >0$. This parameter captures the attractiveness of product $i$, meaning that a high  preference weight indicates a high popularity. Without loss of generality, we assume by convention that the no-purchase preference weight is normalized to $v_0=1$, and we define the total weight of an assortment $S\subseteq \Nc$ as $V(S) \coloneqq \sum_{i \in S} v_i$.
Under the MNL model, if we offer  an assortment $S \subseteq {\cal N}$ to a given customer, she chooses each product $i\in S$ with  probability 
 $$\phi(i, S) \coloneqq \frac{v_i}{1 + V(S)},$$
 \black{with convention that $\phi(i, S) = 0$ for all $i\in \Nc\setminus S$.}
 We refer to $\phi(i,S)$ as the choice probability of product $i$ given the assortment $S$. 
 Alternatively, this customer may decide to not purchase any product, which happens with the complementary probability 
  $$\phi(0, S) \coloneqq \frac{1}{1 + V(S)}.$$
Letting $R(S)$ be the expected  revenue obtained by offering the assortment $S$, we have 

$$R(S) \coloneqq   \sum_{i \in S} p_i \phi(i,S) = \frac{\sum_{i \in S} p_i v_i}{1 + \sum_{i \in S} v_i}.$$



\noindent
{\bf Assortment optimization with visibility constraints.} 
Given a finite stream of $T$ customers, our solution concept consists of offering an assortment $S_t$ to each customer $t\in [T]$. As previously explained, customers make their purchasing choices according to the same MNL model, meaning that each customer independently decides to purchase product $i\in S_t$ with probability $\phi(i, S_t)$, or chooses the no-purchase option, with  probability $\phi(0, S_t)$. As such, the expected revenue we obtain from customer $t$ is $R(S_t)$.
To ensure visibility, we impose constraints that require each product $i \in \mathcal{N}$ to be shown to at least $\ell_i$ customers, assuming that the parameters $\{\ell_i\}_{i\in \Nc}$ are exogenous, with $\ell_i \leq T$.
Our objective is to determine, \black{in an offline fashion}, a collection of assortments $S_1, \ldots, S_T$ in order to maximize the total expected revenue while satisfying these visibility constraints. We refer to this problem as the {\em assortment optimization problem with visibility constraints}  (\ref{APV}), succinctly formulated as follows:

\begin{equation}
\label{APV}
\begin{aligned}
 \max_{ S_1, \ldots, S_T \subseteq \mathcal{N}}  & \; \;      \sum_{t=1}^T  R(S_t)   \\  
  \mathrm{s.t.} \;\;   & \;\;  \sum_{t=1}^T  \mathbbm{1}(i \in S_t) \geq \ell_i, \;\;\; \forall i \in \mathcal{N}
\end{aligned}
\tag{\sf{APV}}
\end{equation}

    \section{Polynomial Time Algorithm for \ref{APV}} \label{sec:apv}

The primary contribution of this section consists of developing a polynomial time algorithm for \ref{APV}. To this end, in Section \ref{subsection:expanded}, we introduce the ``expanded revenue" of an assortment  and its ``expanded set". 
In Section \ref{subsection:prop}, we present a polynomial time procedure for computing these two objects and establish the monotonicity and supermodularity of the expanded revenue function. Leveraging these properties, in Section \ref{subsection:algo}, we characterize the structure of optimal solutions to \ref{APV}, showing that they can be computed in $O(n + T)$ time. Finally, in Section \ref{subsectin:lp}, we prove that \ref{APV} can be formulated as a compact linear program.





\subsection{Expanded Revenue and Expanded Set} \label{subsection:expanded}

We begin our analysis by examining \ref{APV} in the context of a single customer. In this simple scenario, the underlying visibility constraints are given such that each product $i\in \Nc$ either must be displayed (i.e., $\ell_i = 1$) or can be displayed (i.e., $\ell_i=0$). Let $A=  \{i\in \Nc\,\colon\,\ell_i = 1\}$ denote the subset of all products that must be displayed. Consequently, in this simple setting, \ref{APV} is the problem of identifying the assortment that maximizes revenue while including $A$. This particular problem will serve as  the building block for our analysis, as it lays the foundation for understanding the general case involving  $T$ customers. Thus, it leads us to introduce the next two definitions.



\begin{definition}[\bf Expanded revenue]
    \label{expanded revenue}
    Let $A \subseteq \mathcal{N}$. The expanded revenue $\overline{R}(A)$ of $A$ is defined as the maximum expected revenue achieved by any assortment that contains $A$. Namely,

    \begin{equation} \label{eq:exrev}
        \overline{R}(A) \coloneqq \max_{S \subseteq \mathcal{N} : A \subseteq S} R(S).
    \end{equation}
\end{definition}

An optimal solution to the maximization problem in \eqref{eq:exrev} is referred to as the expanded set of $A$. When multiple optimal solutions exist, we break ties by selecting one with the largest cardinality, which is shown to be unique in Lemma \ref{Compute expanded set}, hence proving that the expanded set is well defined. Formally, we have the following definition. 

\begin{definition}[\bf Expanded set]
    \label{expanded set}
The expanded set $\overline{A}$ of $A$ is defined as the assortment within $\mathcal{N}$ that  maximizes the expected revenue  among all assortments containing $A$.  If multiple assortments attain this optimum, $\overline{A}$ is selected as the one of largest cardinality.  In other words,
    $$\overline{A} \coloneqq \underset{S \subseteq \mathcal{N} : A \subseteq S}{\arg \max}   \left\{ |S| \; : \; R(S)=  \overline R (A)   \right\}.   $$
\end{definition}




Problem \eqref{eq:exrev} can be viewed as equivalent to \ref{APV} when considering a single customer scenario $(T=1)$ and defining $A$ as the set of products that must be shown. Thus, $\overline A$ represents the optimal assortment, specifically one with the largest cardinality.

In our  analysis, we consider $\overline{R}$ as a set function that takes an assortment $A \subseteq \mathcal{N}$ as input and returns $\overline{R}(A)$, noting that $\overline{R}(A) = R(\overline{A})$.


\subsection{Properties of the Expanded Revenue Function} \label{subsection:prop}

In this section, we first show that the expanded revenue and the expanded set of a given assortment can both be computed in polynomial time (Lemma \ref{Compute expanded set}). Then, we prove that the expanded revenue is a non-increasing function and that the expanded set is a non-decreasing function (Lemma \ref{monotonicity}). 
Finally, we argue that the expanded revenue function is supermodular (Lemma \ref{supermodularity}), which is a particularly important property for our algorithm design later on. 

\vspace{2mm}
\noindent
{\bf Computing the expanded revenue and expanded set.}
Recall that, without loss of generality $p_1 \geq \cdots \geq p_n$. We define an assortment $S$ to be price-ordered if $S=\black{[k]}$ for some $0 \leq k \leq n$, essentially prioritizing products with high prices; clearly, there  are only $n+1$ such assortments possible.
Consider an arbitrary assortment $A \subseteq {\cal N}$ and its expanded set $\overline A$. In Lemma \ref{Compute expanded set} below, we prove that $\overline A$ is the union of $A$ and a price-ordered assortment. Since there are only $n+1$ possible price-ordered assortments, it suffices to compute the expected revenue of the assortments $A \cup \black{[k]}$, for each $k \in \black{[n]}$. The expanded set corresponds to the assortment with the highest expected revenue. As previously explained, in the case of multiple assortments with the same maximum revenue, we break ties by selecting the one with the largest cardinality. Thus, the expanded set $\overline A$ can be computed  $O(n)$ time, and the expanded revenue is simply $\overline{R}(A)= R(\overline A)$. The proof of Lemma \ref{Compute expanded set} leverages several structural properties of the MNL-based revenue function that are presented in Appendix \ref{apx1}.

\begin{lemma}
    \label{Compute expanded set}
    For any assortment $A \subseteq \mathcal{N}$, its expanded set is given by  
    $ \overline{A} = A \cup \{i \in \mathcal{N} : p_i \geq \overline{R}(A) \}.$ Furthermore, $\overline{R}(A)$ and $\overline{A}$ can be computed in time $O(n)$.
\end{lemma}
\begin{proof}
    By definition of the expanded set, we have $A \subseteq \overline{A}$. Hence, there exists an assortment $B \subseteq \mathcal{N} \setminus A$, such that $\overline{A} = A \uplus B$, \black{where $\uplus$ denotes the disjoint union of sets}. In what follows, we prove that
    $B = \{i \in \mathcal{N} \setminus A   \; : \; p_i \geq \overline{R}(A) \}.$
    \begin{itemize}
        \item {\em Direct inclusion}: Let $i\in B$, and assume by contradiction that $p_i < \overline{R}(A) = R(A \cup B)$. It is easy to verify that, under the MNL model, when we add a product $j\notin S$ to an assortment $S$, the revenue of this assortment weakly increases if and only of $p_j \geq R(S)$. For completeness, we provide the formal statement and the proof of this result in Lemma \ref{Revenue variations} (see Appendix \ref{apx1}). By this lemma, removing product $i$ from $B$ would strictly increase the expected revenue $R(A)$, which contradicts the optimality of $A \cup B$.
        \item {\em Indirect inclusion: } Let $i \in \mathcal{N} \setminus A$ be a product with $p_i \geq \overline{R}(A)$, and suppose by contradiction that $i \notin B$. By Lemma \ref{Revenue variations}, adding $i$ to $B$ would weakly increase the expected revenue of $A\cup B$. If this increase is strict, it contradicts the optimality of $A \cup B$. If the expected revenue remains unchanged, it contradicts the definition of $\overline{A} = A \cup B$ as the optimal assortment of maximum cardinality. 
        \end{itemize}
    Thus,
        $\overline{A} = A \cup \{i \in \mathcal{N} \setminus A\,\colon\, p_i \geq \overline{R}(A) \}.$
        
    Finally, $\overline{A}$ can be computed in $O(n)$ time. Indeed, starting from $A$, we sequentially add products by decreasing price. At each iteration, we can compute the new expected revenue given the previous one in constant time by storing the current numerator and denominator, since we only need to add $p_i v_i$ to the former and $v_i$ to the latter. Finally, we pick the highest revenue assortment among the $n$ computed sets.
\end{proof}


\begin{lemma}[\bf Monotonicity]
    \label{monotonicity} 
    If $ A \subseteq B \subseteq \mathcal{N}$, then  $ \overline{A} \subseteq \overline{B}$ and $\overline{R}(A) \geq \overline{R}(B)$.   
\end{lemma}

\begin{proof}
    It is easy to see that when $A \subseteq B \subseteq \mathcal{N}$, every feasible solution for $\max_{S \subseteq \mathcal{N} , \; B \subseteq S} R(S)$ is a feasible solution for $\max_{S \subseteq \mathcal{N} , \; A \subseteq S} R(S)$. Hence, $\overline{R}(A) \geq \overline{R}(B)$. It follows that $\{i \in \mathcal{N}, p_i \geq \overline{R}(A) \} \subseteq \{i \in \mathcal{N}, p_i \geq \overline{R}(B) \},$ and therefore, by Lemma \ref{Compute expanded set}, we have $$\overline{A} = A \cup \{i \in \mathcal{N}, p_i \geq \overline{R}(A) \} \subseteq B \cup \{i \in \mathcal{N}, p_i \geq \overline{R}(B) \} = \overline{B}.$$
\end{proof}


\begin{lemma}[\bf Supermodularity]
    \label{supermodularity}
    The expanded revenue function $\overline{R}$ is  supermodular, i.e.,
     \black{for every $A\subseteq B \subseteq \mathcal{N}$ and $i\in \Nc\setminus B$,
     $$\overline{R}(B\cup\{i\}) - \overline{R}(B) \geq \overline{R}(A\cup\{i\}) - \overline{R}(A).$$}
\end{lemma}
\begin{proof}

Let $A,B\subseteq \Nc$ be a pair of assortments with $A\subseteq B$, and let $i\in \Nc\setminus B$. The proof of this result is separated into two steps. First, we show that\begin{equation}\label{eq:step1}
    R\left(\overline{B}\right) - R\left(\overline{B}\cup \overline{A\cup\{i\}}\right) \leq R\left(\overline{A}\right) - R\left( \overline{A\cup\{i\}}\right). 
\end{equation}
Subsequently, we show that\begin{equation}\label{eq:step2}
    R\left(\overline{B}\cup \overline{A\cup\{i\}}\right)- R\left(\overline{B\cup \{i\}}\right) \leq 0.
\end{equation}
The result directly follows by summing inequalities \eqref{eq:step1} and \eqref{eq:step2}.
\paragraph{Proving inequality \eqref{eq:step1}.} 
Let us start with the following claim, \black{whose proof appears in Appendix \ref{apx:computation}}.
\begin{claim}\label{cl:computation}
    For any pair of assortments $S_1, S_2\subseteq \Nc$ with $S_1\subseteq S_2$, we have\begin{equation*}
        R(S_1) - R(S_2) = \frac{1}{1+V(S_2)}\cdot \sum_{j\in S_2\setminus S_1}(R(S_1) - p_j)\cdot v_j.
    \end{equation*}
\end{claim}
\noindent Using this claim, we have
\begin{equation*}
    R\left(\overline{B}\right) - R\left(\overline{B}\cup \overline{A\cup\{i\}}\right) = \frac{1}{1+V
        (\overline{B}\cup \overline{A\cup\{i\}})}\cdot \sum_{j\in \overline{A\cup \{i\}}\setminus \overline{B}} \left(R\left(\overline B\right) - p_j\right)v_j.
\end{equation*}
Next, we have $R(\overline{B}) \geq p_j$ for all $j \notin \overline{B}$, since if there exists some $j\notin \overline{B}$ such that $p_j>R(\overline{B})$, adding this product to $\overline{B}$ would increase its expected revenue according to Lemma \ref{Revenue variations}, which would contradict the definition of the extended set of $B$. Therefore, since $V(\overline{B}\cup \overline{A\cup\{i\}})\geq V(\overline{A\cup\{i\}})$, we have
\begin{align}
    R\left(\overline{B}\right) - R\left(\overline{B}\cup \overline{A\cup\{i\}}\right) &\leq \frac{1}{1+V
        (\overline{A\cup\{i\}})}\cdot \sum_{j\in \overline{A\cup \{i\}}\setminus \overline{B}} \left(R\left(\overline B\right) - p_j\right)v_j \notag\\
        & \leq \frac{1}{1+V
        (\overline{A\cup\{i\}})}\cdot \sum_{j\in \overline{A\cup \{i\}}\setminus \overline{B}} \left(R\left(\overline A\right) - p_j\right)v_j,\label{eq:smallsum}
\end{align}
where \black{Equation~\eqref{eq:smallsum}} follows from Lemma \ref{monotonicity}.
Finally, \black{since $\overline{A}\subseteq \overline{B}$ by again Lemma~\ref{monotonicity},} we have $$\left(\overline{A\cup \{i\}}\setminus \overline{B} \right)\subseteq \left(\overline{A\cup \{i\}}\setminus  \overline{A}\right).$$
Moreover, for every $j \in \overline{A\cup \{i\}}\setminus  \overline{A}$, $j\notin \overline{A}$, and in particular, $p_j\leq R(\overline{A})$. Therefore, by adding the missing terms to the sum in Equation \eqref{eq:smallsum}, we have
    \begin{align*}
        R\left(\overline{B}\right) - R\left(\overline{B}\cup \overline{A\cup\{i\}}\right) & \leq \frac{1}{1+V
        \left(\overline{A\cup\{i\}}\right)}\cdot \sum_{j\in \overline{A\cup \{i\}}\setminus \overline{A}} \left(R\left(\overline A\right) - p_j\right)v_j\\
        & = R\left(\overline{A}\right) - R\left(\overline{A\cup \{i\}}\right),
    \end{align*}
where the equality follows from Claim \ref{cl:computation}.
\paragraph{Proving inequality \eqref{eq:step2}. }On the one hand, we have $\{i\}\subseteq\overline{A\cup \{i\}}$. Therefore, we have in particular ${\{i\} \subseteq \overline{B}\cup\overline{A\cup \{i\}}}$. On the other hand,  $B\subseteq \overline{B}$ and therefore $B\subseteq \overline{B}\cup\overline{A\cup \{i\}}$. Hence $B\cup \{i\}\subseteq \overline{B}\cup\overline{A\cup \{i\}}$. Recalling that $\overline{B\cup \{i\}}$ is by definition the maximum revenue assortment containing $B\cup \{i\}$, we have$$
    R\left(\overline{B\cup \{i\}}\right) \geq R\left(\overline{B}\cup\overline{A\cup \{i\}}\right),
$$

\end{proof}

\subsection{Computing Optimal Assortments} \label{subsection:algo}

In this section, we characterize the structure of optimal assortments for \ref{APV}, relying on the supermodularity property of the expanded revenue function. Moreover, we explain how to compute such an assortment in $O(n+T)$ time, leading to an exact polynomial-time algorithm.

\vspace{2mm}
\noindent
{\bf Optimal solution.} Consider an arbitrary \ref{APV} instance, and recall that $\ell_i$ is the minimum number of customers for which we must offer each product $i$. For $t \in \black{\{0\}\cup[T]},$ we define 
\begin{equation}\label{eq:theells}
    L_t = \{i \in \mathcal{N}:\, \ell_i = t \},
\end{equation}
Noting that $(L_t)_{0 \leq t \leq T}$ is a partition of $\cal N$, our candidate solution for \ref{APV} is given by offering each customer $t\in [T]$ the set of products
\begin{equation} \label{eq:sol}
    {S_t^*} = \overline{\bigcup_{t \leq u \leq T} L_u}.
\end{equation}

Since  $ \bigcup_{t+1 \leq u \leq T} L_u  \subseteq  \bigcup_{t \leq u \leq T} L_u$, the monotonicity property in Lemma \ref{monotonicity} implies that $S_{t+1}^* \subseteq S_{t}^*$ for any $t=0,\ldots,T-1$. Therefore, our solution has a nested structure, where
$ S_T^* \subseteq S_{T-1}^*\subseteq \cdots \subseteq S_1^*$. In the following, we prove that these assortments are optimal for \ref{APV}, and that they can be computed in polynomial time. Indeed, Lemma \ref{Compute expanded set} shows that each $S_t^*$ can be computed in $O(n)$ time , meaning that the entire solution can be computed in $O(nT)$ time. We can further improve this running time to $O(n+T)$ as shown below.



\begin{theorem}{}
    \label{Solution structure}
    The  sequence of assortments $(S_t^*)_{1 \leq t \leq T}$ defined in \eqref{eq:sol} is optimal for \ref{APV}. 
    Moreover, such a solution can be computed in $O(n+T)$ time.
\end{theorem}

\begin{proof}
    We prove the result by induction. First, when $T= 1$, for any product $i\in \Nc$, we either have $\ell_i=0$ or $\ell_i=1$. Noting that $L_1$ is the set of products $i$ for which $\ell_i=1$, \ref{APV} reduces to the problem of computing the optimal assortment that contains $L_1$, i.e., $
        \max_{S\subseteq \Nc\colon L_1\subseteq S}R(S),
    $
    whose solution is $\overline{L_1}$ by definition of the expanded set. By noticing that $S_1^* =\overline{L_1}$, we deduce that $S_1^*$ is the optimal solution to \ref{APV}.
    
    Now consider $T\geq 2$, and assume by induction that for any set of visibility constraints, the optimal solution of \ref{APV} for $T-1$ customers is given by the assortments defined in Equation \eqref{eq:sol}. Let us show that the latter assumption holds for $T$ customers.
    We denote by $A$ the set of all products that must be shown to at least one customer, i.e., $A\coloneqq \bigcup_{t=1}^TL_t$. We start by providing the following crucial intermediary claim, whose proof is presented on Page~\pageref{prf:intermediary}.
    \begin{claim}\label{cl:intermediary}
        There exists an optimal solution $S_1, \ldots, S_T$ to \ref{APV} with $S_1 =S_1^*$.
    \end{claim}
    In other words, Claim \ref{cl:intermediary} states that there exists an optimal solution $S_1,\ldots, S_T$ to \ref{APV} such that $S_1$ contains all products that must be shown at least once, i.e., those with $\ell_i\geq 1$. Consequently, this property allows us to focus only on feasible solutions of \ref{APV}, that offer $S_1^*$ to customer $1$. Maximizing the revenue amongst these solutions thereby guarantees attaining the optimal objective, and solving \ref{APV} is in turn equivalent to solving the following optimization problem:
    \begin{equation*}
        \begin{aligned}
         \max_{ S_2, \ldots, S_T \subseteq \mathcal{N}}  & \; \;      R(S_1^*)+\sum_{t=2}^T  R(S_t)   \\  
          s.t. \;\;   & \;\;  1+\sum_{t=2}^T  \mathbbm{1}(i \in S_t) \geq \ell_i, \;\;\; \forall i \in A,
        \end{aligned}
        \end{equation*}
    which itself reduces to the following instance of \ref{APV}:
    \begin{equation}\label{eq:reducedinstance}
        \begin{aligned}
         \max_{ S_2, \ldots, S_T \subseteq \mathcal{N}}  & \; \;      \sum_{t=2}^T  R(S_t)   \\  
          s.t. \;\;   & \;\;  \sum_{t=2}^T  \mathbbm{1}(i \in S_t) \geq \Tilde \ell_i, \;\;\; \forall i \in \Nc,
        \end{aligned}
    \end{equation}
    where $\Tilde \ell_i = \ell_i - \mathbbm 1(i\in A)$, for all $i\in \Nc$. Noting that this last problem is an instance of \ref{APV} with $T-1$ customers, we can apply the induction hypothesis, which implies that the optimal solution is given by Equation \eqref{eq:sol}. Directly applying this formula implies that $S_2^*, \ldots, S_T^*$ is an optimal solution to \eqref{eq:reducedinstance}, and hence that $S_1^*,\ldots,S_T^*$ is an optimal solution for \ref{APV}.
    \paragraph{Running time. }The running time of our algorithm for APV can be improved from $O(nT)$ to $O(n+T)$ as follows. We start by computing $S_T^*$. Then, proceeding by order of decreasing customer index, when computing each $S_t^*$, we do not need to directly compute $\llbracket\bigcup_{t \leq u \leq T} L_u\rrbracket$. Instead, since $S_t^*\supseteq S_{t+1}^* $, we should only check the products $\{i \in \mathcal{N}\backslash (S_{t+1}^* ) : p_i \geq \overline{R}(S_{t}^* ) \}$\black{. Since the latter set is price ordered, we can stop checking when reaching the first item we do not include in $S^*_t$. In the end, $|S_t^*|-|S_{t+1}^*|+1$ items are checked (including the first one that is not included), eventually leading to a running time of 
$$O(1)\cdot \left(1+|S_T^*|+\sum_{t=1}^{T-1} \left(1+|S_t^*|-|S_{t+1}^*|\right)\right)=O(|S_1^*|+T)= O(n+T).$$
}
    \paragraph{Proof of Claim \ref{cl:intermediary}.\label{prf:intermediary} }The proof of this result mainly relies on exploiting the supermodularity of $\overline{R}$, as stated in Lemma \ref{supermodularity}. \black{We start by showing that there exists a solution $S_1,\ldots, S_T$ to \ref{APV} such that $S_1\supseteq A$.} Assume by contradiction that there is no optimal solution of \ref{APV} with $S_1 \supseteq A$. Let $\hat S_1, \ldots, \hat S_T$ be the optimal solution of \ref{APV} that maximizes $|\hat S_1\cap A|$; in the case of ties, we pick an arbitrary solution that maximizes $|\hat S_1\cap A|$. Since $A \nsubseteq \hat S_1$, there exists some product $j\in A\setminus S_1$. Moreover, since $j\in A$, we know that $\ell_j\geq 1$, and therefore, this product must be shown to at least one customer, which means that there exists some $u\in [T]\setminus \{1\}$ with $j\in \hat S_u$. We define the following new solution to \ref{APV}: $S_1 = \hat S_1\cup \hat S_u$, $S_u = \hat S_1\cap \hat S_u$, and $S_t = \hat S_t$ for all $t\notin \{1,u\}$. First, this newly defined solution is also feasible since any product offered once in either $\hat S_1$ or $\hat S_u$ \black{(exclusive)} is also shown in $S_1$, and each product shown in both $\hat S_1$ and $\hat S_u$ is also shown in both $S_1$ and $S_u$. Second, $S_1,\ldots, S_T$ must be an optimal solution. Indeed, by the supermodularity\footnote{\black{Here, we use the equivalent definition of supermodularity: A function $f:\Omega \rightarrow \R$ is supermodular if ${f(A\cup B)} + f(A\cap B) \geq f(A)+ f(B)$ for all $A,B\subseteq \Omega$.}} property, we have $$
        R\left(\hat S_1\cup \hat S_u\right)+ R\left(\hat S_1\cap \hat S_u\right) \geq R\left(\hat S_1\right)+ R\left(\hat S_u\right).
    $$
    Therefore,$$
        \sum_{t=1}^TR\left(S_t\right) \geq \sum_{t=1}^TR\left(\hat S_t\right).
    $$
    Third, we have $|S_1\cap A| \geq |\hat S_1\cap A|+1$ since $(\hat S_1\cap A)\subsetneq (S_1\cap A)$, as the latter set contains $j$ whereas the former does not. This contradicts the definition of $\hat S_1,\ldots, \hat S_T$, as the optimal solution that maximizes $|\hat S_1\cap A|$, and thereby proves by contradiction that there exists a solution $S_1,\ldots, S_T$ with $S_1 \supseteq A$.
    
    Finally, we show that we can in particular take $S_1 = S_1^*$. On the one hand, the solution $A, S_2, \ldots, S_T$ is also feasible, as it is obtained by removing from $S_1$ all products $j$ such that $\ell_j=0$, which \black{are absent from the visibility constraints and thus cannot violate them}. Therefore, noting that $\overline A = S_1^*$, the solution $S_1^*, S_2, \ldots, S_T$ is also feasible, since $A\subseteq \overline{A}$. Finally, we have\begin{align*}
        R\left(S_1^*\right)+\sum_{t=2}^T R\left(S_t\right) &= R\left(\overline A\right)+\sum_{t=2}^T R\left(S_t\right)\geq \sum_{t=1}^TR(S_t),
    \end{align*}
    where the latter inequality follows by definition of the expanded set of $A$, and the fact that $A\subseteq S_1$\black{, giving $R(\overline{A})\geq R(\overline{S_1})\geq R(S_1)$}. In conclusion, there exists an optimal solution $S_1,\ldots, S_T$ of \ref{APV} with $S_1 = S_1^*$.
\end{proof}


\subsection{Linear Programming Formulation} \label{subsectin:lp}
\vspace{-2mm}
In this section, we provide a linear programming-based exact algorithm for solving \ref{APV}. For this purpose, consider the unconstrained assortment problem under the MNL model for a single customer, given by
\begin{equation}
\label{Unconstrained problem}
\begin{aligned}
\max_{S \subseteq \mathcal{N}} \quad  R(S),
\end{aligned}
\tag{\sf{AP}}
\end{equation}
where we recall that $R(S)$ is the expected revenue of the assortment $S$ in this model.
It is well-known that 
\ref{Unconstrained problem} can be formulated as the following LP (\cite{Gallego2011AGA}):
\begin{equation*}
    \begin{aligned}
     \max_{(\alpha_i)_{0 \leq i \leq n}} \;\; & \sum_{i=1}^n p_i \alpha_i \\
     s.t. \quad &  0 \leq \frac{\alpha_i}{v_i} \leq \alpha_0, &&\quad\forall i \in \mathcal{N}, \\ 
     & \sum_{i=0}^n \alpha_i = 1.
    \end{aligned}
\end{equation*}

Motivated by the structure of our optimal solution of \ref{APV}, as prescribed by Equation \eqref{eq:sol},
we propose the following linear formulation for  \ref{APV}:
  \begin{equation}\label{LP}\tag{\sf LP}
    \begin{aligned}
     \max_{\boldsymbol{\alpha}} \;\; & \sum_{i=1}^n p_i \sum_{t=1}^T \alpha_i^t \\
     s.t. \quad & \text{(A)}\quad&&  \sum_{i=0}^n \alpha_i^t = 1, &&\quad\forall t \in [T], \\ 
     &\text{(B)}\quad&& \frac{\alpha_i^t}{v_i} = \alpha_0^t, &&\quad\forall i \in \mathcal{N}, \;\; \forall t \in \black{[\ell_i]},\\
     &\text{(C)}\quad&&  0 \leq \frac{\alpha_i^t}{v_i} \leq  \alpha_0^t ,&&\quad\forall i \in \mathcal{N}, \;\; \forall t \in \black{[T]\setminus[\ell_i]}.
    \end{aligned}
    \end{equation}

\black{
\noindent In the above linear program, we invite the reader to think of the variable $\alpha_i^t$ as the probability that customer $t$ purchases product $i$, when optimal assortments are offered to customers. Our optimality proof later shows that these variables indeed correspond to the stated probabilities. Using this interpretation, the objective of \ref{LP} represents the expected revenue collected from our collection of customers. Constraint (A) ensures that every customer either purchases a product or opts for the no-purchase option, bringing the sum of probabilities of these events to $1$. For each product $i\in \Nc$, Constraint (B) links the probability that a customer $t\in[\ell_i]$ purchases product $i$ to the probability of leaving the market. In essence, it imposes that product $i$ is included in assortment $S_t$ when $t\in[\ell_i]$, which we know is optimal by virtue of Equation \eqref{eq:sol}. Finally, Constraint (C) bounds the choice probability of each product $i$ by customer $t \in [T]\setminus [\ell_i]$. We later show that using the properties of basic feasible solutions of linear programs that $\alpha_i^t$ adheres to either the upper or the lower bound, corresponding to either including product $i$ in assortment $S_t$ or excluding it.

In what follows, we use $\opt^{LP}$ and $\opt^{APV}$ to denote the optimum values of \ref{LP} and \ref{APV} respectively.
}

\begin{theorem} \label{thm:LPformulation}
    $\opt^{LP}=\opt^{APV}$. Moreover, given an optimal solution $\boldsymbol{\alpha}^*$ for \ref{LP}, we can construct an optimal solution $S_1^*, \ldots, S_T^*$ for \ref{APV} by taking $
        S_t^* = \{i\in \Nc\colon \alpha_i^t>0\}
    $
    for all $t\in [T]$.
\end{theorem}

\begin{proof}
     On the one hand, we prove that $\opt^{LP} \geq \opt^{APV}$, by constructing a feasible solution to \ref{LP} whose objective is at least $\opt^{APV}$. On the other hand, using the same technique, we argue that $\opt^{LP} \leq \opt^{APV}$.

    \noindent {\em {First inequality: $\opt^{LP} \geq \opt^{APV}$}. } Recall that $(S_1^*, \ldots,S_T^*)$ stands for the optimal solution of \ref{APV}, whose definition is stated in Equation \eqref{eq:sol}. We introduce a solution $\boldsymbol{\alpha}$ for \ref{LP} such that for every customer $t\in [T]$ and every option $i\in \Nc \cup \{0\}$, \begin{equation*}
        \alpha_i^t = \phi(i, S_t^*),
    \end{equation*}
    where \black{we recall that} $\phi(i,S) = 0$ for every assortment $S\subseteq \Nc$ and every product $i\in \Nc\setminus S$ by convention.
    Let us show that $\boldsymbol \alpha$ is feasible. Constraints (A) and (C) are straightforward. Indeed, for Constraint (A), we have for all $t\in [T]$, \begin{equation*}
        \sum_{i=0}^n\alpha_i^t = \frac{1}{1+V(S_t^*)} + \sum_{i\in S_t^*}\frac{v_i}{1+V(S_t^*)} =1.
    \end{equation*}
    Constraint (C) is also directly satisfied by construction, as for all $i\in \Nc$ and $t\in [T]$, we have
    \begin{equation*}
        \black{0\leq}\; \alpha_i^{t} = \frac{v_i}{1+V(S_t^*)}\cdot \mathbbm 1\left(i\in S_t^*\right) \leq \frac{v_i}{1+V(S_t^*)} = v_i\alpha_{0}^t.
    \end{equation*}
    Regarding Constraint (B), let $i\in \Nc$ and $t\in \black{[\ell_i]}$. Then
    \black{\begin{align}
        i&\in L_{\ell_i} \label{eq:1}\\ & \subseteq \bigcup_{t\leq u\leq T}L_u \label{eq:2}\\ &\subseteq \overline{\bigcup_{t\leq u\leq T}L_u}\label{eq:3}\\ & = S_t^*\label{eq:4}.
    \end{align}}
    \black{Here, Equation~\eqref{eq:1} comes from the} definition of $L_{\ell_i}$ (see Equation \eqref{eq:theells}). \black{Equation~\eqref{eq:2}} is a consequence of $t\leq \ell_i\leq T$. \black{Equation~\eqref{eq:3}} follows from the fact that the extended set of any assortment $S\subseteq \Nc$ contains $S$. \black{Finally, Equation \eqref{eq:4}} is by definition of $S_t^*$, as given by Equation \eqref{eq:sol}.
    Therefore, $$
        \alpha_{i}^t = \frac{v_i}{1+V(S_t^*)} = v_i\alpha_0^t.
    $$
    Finally, the objective value of $\boldsymbol{\alpha}$ is precisely
    \begin{equation*}
        \sum_{i=1}^n p_i\sum_{t=1}^T \alpha_i^t = \sum_{t=1}^T\sum_{i\in S_t^*} p_i \phi(i,S_t^*) = \sum_{t=1}^TR(S_t^*) = \opt^{APV}.
    \end{equation*}
    Since the objective value of $\boldsymbol{\alpha}$ forms a lower bound on $\opt^{LP}$, we deduce that  ${\opt^{LP} \geq \opt^{APV}}$.

    \noindent {\em {Second inequality: $\opt^{LP} \leq \opt^{APV}$}. } Consider an optimal basic feasible solution $\boldsymbol{\alpha}$ to \ref{LP}. Since \ref{LP} is a linear program with $T\cdot(n+1)$ variables, any basic solution has at least $T\cdot(n+1)$ active constraints. Since \ref{LP} already contains $T+\sum_{i\in \Nc}\ell_i$ equality constraints, $\boldsymbol{\alpha}$ activates at least $\sum_{i\in \Nc}(T-\ell_i)$ constraints from the remaining $2\cdot \sum_{i\in \Nc}(T-\ell_i)$ inequality constraints. Moreover, \black{for each pair of constraints $\alpha_i^t \geq 0$ and $\alpha_i^t\leq v_i\alpha_0^t$, at most one can be satisfied with equality. Therefore, at most $\sum_{i\in \Nc}(T-\ell_i)$ inequality constraints can be active.} Thus, $\boldsymbol\alpha$ activates exactly $\sum_{i\in \Nc}(T-\ell_i)$ inequality constraints, and we have $\alpha_{i}^t \in \{0, v_i\alpha_0^t\}$, for all $i\in \Nc$ and $t\in \black{[T]\setminus[\ell_i]}$. Using this new observation, we construct a feasible solution $(S_1,\ldots, S_T)$ for \ref{APV} as follows: For all $t\in [T]$, let $S_t \coloneqq \{i\in \Nc\,\colon\, \alpha_i^t\neq 0\}$. In particular, if $i \in S_t$, then $\alpha_i^t = v_i\alpha_0^t$. It is easy to verify that $\alpha_i^t$ is exactly the choice probability of product $i$ out of the assortment $S_t^*$. Indeed, for all $t\in [T]$, we have$$
        \alpha_0^t+\sum_{i\in \Nc} \alpha_i^t= \alpha_0^t+ \sum_{i\in S_t}v_i \alpha_0^t= \alpha_0^t\cdot (1+V(S_t)),
    $$
    which implies that $\alpha_i^0 = \phi(0,S_t)$ using Constraint (A) of \ref{LP}, and thereby that $\alpha_i^t = \phi(i, S_t)$ for all $i \in S_t$.
    Next, we can easily see that $(S_1, \ldots, S_T)$ is feasible for \ref{APV}, as any product $i\in \Nc$ is included in \black{each of} the first $\ell_i$ assortments, which satisfies the visibility constraints. Finally, we have$$
        \opt^{LP} = \sum_{i\in \Nc}p_i \sum _{t=1}^T\alpha_i^t = \sum_{t=1}^T\sum_{i\in \Nc}p_i\phi(i,S_t) = \sum_{t=1}^TR(S_t) \leq \opt^{APV}.
    $$
\end{proof}

 \section{\ref{APV} with Cardinality Constraints}
    \label{sec:APVC}
In this section, we consider the cardinality constrained version of our fundamental model, referred to as {\em assortment optimization with visibility and cardinality constraints (\ref{APVC})}. Here, each customer can be shown at most $k$ products, where $k$ is a positive integer specified as part of the input description. Formally, \ref{APVC} is defined as follows: 
\begin{equation}
\label{APVC}
\tag{\sf{APVC}} 
\begin{aligned}
 \max_{ S_1, \ldots, S_T \subseteq \mathcal{N}} & \; \; \sum_{t=1}^T  R(S_t)  \\ \mathrm{s.t.} \;\; &\sum_{t=1}^T  \mathbbm{1}(i \in S_t) \geq \ell_i, \;\;\; &&\forall i \in \mathcal{N}, \\ & |S_t| \leq k, && \forall t \in [T].
\end{aligned}
\end{equation}

\paragraph{Hardness results.} First, we study the \black{computational} complexity of \ref{APVC}, showing that unlike its \black{polynomially-solvable} unconstrained counterpart, \ref{APVC} is strongly NP-hard.
\begin{theorem} \label{NP-hardness} \ref{APVC} is strongly NP-hard, even when all prices are equal. Moreover, \black{\ref{APVC} does not admit an FPTAS, even with equal prices}, unless $\mathrm{P} = \rm{NP}$.
\end{theorem}
The proof of this \black{result} relies on a reduction \black{from} the \texttt{3-PARTITION} problem, where the objective is to \black{decide whether a set of $3T$ positive integers can be partitioned} into $T$ triplets, such that the sum of elements within each triplet is identical. The full proof is \black{given in} Appendix \ref{apxhard}.
To the best of our knowledge, within the assortment optimization literature, there is only one other prominent choice model for which the unconstrained variant is solvable in polynomial time, while the cardinality constrained variant is NP-hard, which is the Markov chain model \citep{blanchet2016markov, desir2020constrained}. It is worth noting that despite this similarity, \black{under the Markov Chain choice model}, the cardinality constrained assortment optimization problem is APX-hard even with equal prices \black{\cite[Theorem~5]{desir2020constrained}}, whereas \ref{APVC} is strongly NP-hard, and admits a PTAS when prices are equal, as we proceed to show next.
\paragraph{Approximation scheme.} In the remainder of this section, we focus on the question of designing a polynomial time approximation scheme (PTAS), for the special case of equal prices. This setting corresponds to a sales maximization \black{(or market share)} problem, which is particularly useful when the platform or the company's goal is to maximize the captured portion of customers (see, e.g., \cite{desir2020constrained, derakhshan2022product, asadpour2023sequential}). \black{In particular, we design an algorithm which outputs a random solution for \ref{APVC}, whose objective value is, in expectation, within $1-\eps$ of optimal.}

\begin{theorem} \label{thm:PTAS}
For any $\eps>0$, \ref{APVC} with equal prices can be approximated within factor $1-\eps$ of optimal. The running time of our algorithm is $O(n^3T^{2^{O(\frac{1}{\eps}\log^2(\frac{1}{\eps}))}})$.
\end{theorem}
In Section \ref{subsec:discret}, we show that any \ref{APVC} instance with equal prices can be \black{approximated by} a ``friendlier" instance, \black{where the weights of a particular subset of products are slightly altered}. Importantly, we only incur an $\eps$-loss due to this \black{approximation}. In Section \ref{subsec:linearsol}, \black{we write our approximate problem as an integer linear program} through a carefully crafted guessing procedure. In Section \ref{subsec:GKPS}, we elaborate on a dependent rounding scheme introduced by \cite{gandhi2006dependent}, which we subsequently utilize in Section \ref{subsubsec:bipartite} to obtain approximate solutions for the \black{above-mentioned integer linear program}, thereby completing the description of our PTAS. Finally, in Section \ref{subsec:analysis}, we show that our algorithm indeed computes a $(1-\eps)$-approximation.
\subsection{Step 1: Discretizing the Universe of Products}\label{subsec:discret}
By specializing \ref{APVC} to the setting with identical prices, our current model can be written as
\begin{equation}
\label{eq:SPVC}\tag{\sf{SPVC}}
\begin{aligned}
 \max_{ S_1, \ldots, S_T \subseteq \mathcal{N}} & \; \; \sum_{t=1}^T\frac{V(S_t)}{1+V(S_t)}  \\ \mathrm{s.t.} \;\; &\sum_{t=1}^T  \mathbbm{1}(i \in S_t) \geq \ell_i, \;\;\; &&\forall i \in \mathcal{N}, \\ & |S_t| \leq k, && \forall t \in [T].
\end{aligned}
\end{equation}

\paragraph{Weight rounding.} In what follows, we reduce an arbitrarily-structured \ref{eq:SPVC} instance  into a modified instance, by rounding down the \black{preference} weights of a particular subset of products. Then, \black{this alteration will be shown to incur} at most an $\eps$-loss.
For any product $i\in \Nc$, \black{we operate according to three cases:}\begin{itemize}
    \item If $v_i\geq 1/\eps$, then $v_i$ is rounded down to $\nota{v}_i = 1/\eps$.
    \item If $v_i< \eps^5$, then $v_i$ remains unchanged and we define $\nota{v}_i = v_i$.
    \item If $\eps^5\leq v_i < 1/\eps$, then $v_i$ is rounded down as follows: Let $\black{q_i}\geq 1$ be the unique integer for which $\eps^5\cdot(1+\eps)^{\black{q_i}-1} \leq v_i < \eps^5\cdot(1+\eps)^{\black{q_i}}$. Then $v_i$ is rounded down to $\nota{v}_i = \eps^5\cdot(1+\eps)^{\black{q_i}-1}$. Note that by this definition, we have \begin{equation}\label{eq:weightbound}
        \nota{v}_i \leq v_i\leq (1+\eps)\cdot \nota{v}_i.
    \end{equation}
\end{itemize}
For each $i\in \Nc$, we denote the product associated with the weight $\nota{v}_i$ simply by $\nota{i}$, \black{with $\nota{\Nc} = \{\nota{i}\,\colon\,i\in \Nc\}$} being our new universe of products. We invite the reader to think of the notation $\nota{\cdot}$, as a one-to-one map from the universe $\Nc$ to the modified universe $\nota{\Nc}$, where several weights have been rounded down. Similarly to \black{$V(\cdot)$, we define the total weight of an assortment of products $S\subseteq \nota{\Nc}$ by} $\nota{V}(S) \coloneqq \sum_{\nota{i} \in S} \nota{v}_i$.

\paragraph{Analysis.} We proceed by showing that \black{any approximation guarantee for the resulting instance would migrate to the original instance up to an extra multiplicative loss of $1-\eps$}. To this end, \black{note that the former instance can be written as}
\begin{equation}
\label{eq:SPVCd}\tag{$\widehat{\sf SPVC}$}
\begin{aligned}
 \max_{ S_1, \ldots, S_T \subseteq \nota{\mathcal{N}}} & \; \; \sum_{t=1}^T\frac{\nota{V}(S_t)}{1+\nota{V}(S_t)}  \\ \rm{s.t.} \;\; &\sum_{t=1}^T  \mathbbm{1}(\nota{i} \in S_t) \geq \ell_i, \;\;\; &&\forall \nota{i} \in \nota{\Nc}, \\ & |S_t| \leq k, && \forall t \in [T].
\end{aligned}
\end{equation}
Let $\opt$ and $\widehat\opt$ respectively denote the \black{optimum values of} \ref{eq:SPVC} and \ref{eq:SPVCd}. Additionally, for any sequence of assortments $S_1,\ldots, S_T\subseteq \Nc$, \black{we make use of the shorthand notation}
\begin{equation*}
    \obj(S_1,\ldots, S_T) = \sum_{t=1}^T\frac{V(S_t)}{1+V(S_t)},
\end{equation*}
and \black{similarly,} for any sequence of assortments $S_1,\ldots, S_T\subseteq \nota{\Nc}$, let 
\begin{equation*}
    \widehat{\obj}(S_1,\ldots, S_T) = \sum_{t=1}^T\frac{\nota{V}(S_t)}{1+\nota{V}(S_t)}.
\end{equation*}
In Lemma \ref{lem:epsobjective} below, \black{whose proof is presented in Appendix \ref{apx:epsobjective}}, we show that the objective achieved by any sequence of assortments $S_1,\ldots,S_T\subseteq \Nc$ is within a factor $1\pm\eps$ of the objective achieved by their rounded counterparts $\nota{S}_1,\ldots,\nota{S}_T$, where $\nota{S}_t \coloneqq {\{\nota{i}\,\colon\, i \in S_t\}}$. This claim implies that $\widehat\opt\leq \opt\leq (1+\eps)\cdot\widehat{\opt}$.

\begin{lemma}\label{lem:epsobjective}
    For any sequence of assortments $S_1,\ldots,S_T\subseteq \Nc$, we have $$\widehat{\obj}(\hat{S}_1,\ldots, \hat{S}_T)\leq \obj(S_1,\ldots, S_T)\leq (1+\eps)\cdot\widehat{\obj}(\nota{S}_1,\ldots, \nota{S}_T).$$
\end{lemma}
We defer the proof of this lemma to Appendix \ref{apx:epsobjective}.
In the remainder of this section, we focus on dealing with the reduced instance \ref{eq:SPVCd}, and for simplicity of notation, we denote \black{its underlying products} simply by $1,\ldots, n$. Similarly, their respective weights will be denoted by $v_1,\ldots,v_n$.

\subsection{Step 2: Linearization of the Objective Function}\label{subsec:linearsol}
In this section, we \black{propose a method to linearize} the optimization problem \ref{eq:SPVCd}, by guessing certain parameters pertaining to its optimal solution. We start by presenting our detailed guessing procedure, before leveraging the guessed parameters to construct a linearized formulation.

\paragraph{Customer types.} Let $S_1^*,\ldots, S_T^*$ be an (unknown) optimal solution for our instance \ref{eq:SPVCd}. \black{With respect to this solution}, we aim at categorizing customers \black{in terms of} to the total weight of the products they are offered. First, we say that customer $t\in [T]$ is {\em light} when $V(S_t^*) < \eps$. Similarly, customer $t$ is {\em heavy} when $V(S_t^*)\geq 1/\eps$. Otherwise, $\eps\leq V(S_t^*)<1/\eps$ and we say that this customer is {\em medium}. We further partition medium customers into $L$ classes $G_1, \ldots, G_L$, depending on the value of $V(S_t^*)$. Specifically, \black{each class $G_\ell$ is defined as}
\begin{equation*}
    G_{\ell} = \left\{ t\text{ medium }\,\colon\, \eps\cdot (1+\eps)^{\ell-1}\leq V(S_t^*)< \eps\cdot (1+\eps)^{\ell}\right\}.
\end{equation*}
\black{Here, $L$ is defined as} the smallest integer for which $\eps\cdot(1+\eps)^{L}\geq 1/\eps$, \black{meaning that} $L = O(\frac{1}{\eps}\log\frac{1}{\eps})$.
\black{Indeed, since $\eps\cdot(1+\eps)^{L-1} < 1/\eps$, we have
\begin{equation}\label{eq:Lorder}
L < 1 + \frac{\log{(1/\eps^2)}}{\log(1+\eps)} = O\left(\frac1\eps\log\frac1\eps\right).
\end{equation}
}As such, the classes $G_1,\ldots, G_L$ form a partition of medium customers. Additionally, we denote the sets of light and heavy customers by $G_{\rm light}$ and $G_{\rm heavy}$ respectively. 

\paragraph{Packing patterns.} Next, we proceed by defining classes of products. Recall from Section \ref{subsec:discret} that, for every product $i\in \Nc$ with $v_i \geq \eps^5$, either there exists some $1\leq q\leq Q-1$ for which $v_i = \eps^5\cdot(1+\eps)^{q-1}$ or $v_i = 1/\eps$, where $Q$ is the smallest integer such that $\eps^5\cdot (1+\eps)^{Q-1}\geq 1/\eps$. \black{In particular, we have $Q = O\left(\frac{1}{\eps}\log\frac1\eps\right)$, justified in the same manner as in Equation~\eqref{eq:Lorder}}. For every $1\leq q\leq Q-1$, let $D_q$ denote the set of products in $\Nc$ with $v_i = \eps^5\cdot (1+\eps)^{q-1}$, and let $D_Q$ denote the set of products with $v_i = 1/\eps$. Similarly, let $D_0$ be the class of products $i\in \Nc$ with $v_i<\eps^5$. By construction, the classes $D_0, \ldots D_Q$ form a partition of the universe of products $\Nc$.

\noindent Now, a packing pattern is a $(Q+1)$-dimensional vector, where each coordinate takes \black{one of the values $0,1,2,\ldots, 1/\eps^6, \star\,$}; here, $1/\eps^6$ is assumed to be an integer, without loss of generality. For a given customer $t\in [T]$, each entry $q=0,\ldots,Q$ of her associated packing pattern specifies the number of products from class $D_q$ in the assortment $S_t^*$, with the convention that $\star$ indicates that the number of products from class $D_q$ is strictly greater than $1/\eps^6$.
By this definition, the number of possible packing patterns is  $$\left(\frac{1}{\eps^6}+2\right)^{Q+1} = 2^{O(Q\log\frac{1}{\eps})}   =2^{O(\frac{1}{\eps}\log^2\frac{1}{\eps})}.$$ 

\paragraph{Guessing procedure.} We are now ready to present our guessing procedure. For each class of customers $\ell \in \{{\rm light}, {\rm heavy}, 1,\ldots, L\}$ and for each packing pattern $P$, we guess the number of customers in $G_{\ell}$ whose packing pattern is $P$. Let us designate this guessed value by $K_{\ell,P}$.
Since there are $L+2 = O(\frac{1}{\eps}\log\frac{1}{\eps})$ classes of customers and $O(2^{O(\frac{1}{\eps} \log^2(\frac{1}{\eps}))})$ packing patterns, the number of such pairs is $O(2^{O(\frac{1}{\eps} \log^2(\frac{1}{\eps}))})$.  \black{Due to having} $T$ customers, we have $T+1$ possible guesses for each \black{of these pairs}, meaning that the total number of guesses is $O(T^{2^{O(\frac{1}{\eps}\log^2(\frac{1}{\eps}))}})$.

It is important to note that this guessing procedure allows us to determine, for any customer $t\in [T]$, both her type and her packing pattern in the optimal solution $S_1^*,\ldots,S_T^*$, up to a permutation of $\black{[T]}$. Indeed, since for each class $\ell\in \{{\rm light}, {\rm heavy}, 1,\ldots, L\}$, we know the number of customers that \black{are characterized by} each packing pattern, we know the quantity $|G_\ell|$. Hence, since all customers are initially interchangeable, we can arbitrarily assign each customer to a type according to our guess. Subsequently, given our guess $K_{\ell, P}$ for every customer type $\ell$ and every packing pattern $P$, since the customers within the same type are interchangeable a priori, we arbitrarily assign each customer to a packing pattern, according to our guess. Consequently, for each customer $t\in [T]$, we have determined both her type and her packing pattern, up to a permutation of $\black{[T]}$.

\paragraph{Integer linear programming formulation.}\label{subsec:linearprog}
We are now ready to leverage these guessed parameters in order to introduce our integer linear programming formulation \ref{eq:ILP}. \black{In the sequel, an approximate solution to this formulation} will be shown to be near optimal for \ref{eq:SPVCd}.

Let us start with some useful notation.
For each customer $t$, we define a lower bound $V_t$ on $V(S_t^*)$ as follows. If $t$ is light then $V_t = 0$, and if $t$ is heavy, then  $V_t = 1/\eps$. Similarly, if $V_t$ is of type $\ell\in [L]$, \black{i.e., $t\in G_\ell$}, then $V_t = \eps\cdot (1+\eps)^{\ell-1}$. In addition, we denote the packing pattern of customer $t$ by $(k_{0,t}, k_{1,t},\ldots, k_{Q,t})$.
Using these newly-defined parameters, we formulate the following integer linear program:
\begin{equation}\label{eq:ILP}\tag{ILP}
 \begin{aligned}
\max_{\mathbf{x}} &&& \sum_{t\in G_{\rm light}}\sum_{i\in \Nc}v_ix_{it} \\
\textrm{s.t.}\quad&\text{(I)} &&\displaystyle\sum_{t=1}^T x_{it} \geq  \ell_i,&&\quad  \forall i\in \Nc ,\\
&\text{(II)} &&\sum_{i\in \Nc} x_{it} \leq k,&&\quad  \forall t\in [T],\\
&\text{(III)} && \sum_{i\in D_q}x_{it}  =k_{q,t}, && \quad \forall t\in [T],\,\forall q=0,\ldots,Q \text{ \black{with} } k_{q,t}\neq \star,\\
&\text{(IV)} && \sum_{i\in D_q}x_{it}  \geq \frac{1}{\eps^6}+1,  &&\quad \forall t\in [T],\,\forall q=0,\ldots,Q \text{ \black{with} } k_{q,t}= \star,\\
&\text{(V)} && \sum_{i\in \Nc} v_i x_{it}  \geq  V_{t}, & & \quad\forall t\in [T], \\
&\text{(VI)} && \sum_{i \in \Nc} v_ix_{it}  \leq  \eps, & &\quad\forall t\in G_{\rm light}, \\
& \text{(VII)} &&x_{it}  \in  \{0,1\}, & & \quad\forall i\in \Nc,\, \forall t\in [T].
\end{aligned}
\end{equation}


Here, for each customer $t\in [T]$ and each product $i\in \Nc$, $x_{it}$ is a binary variable indicating whether product $i$ is offered to customer $t$. \black{The first observation is that the objective function only includes light customers. Intuitively, this feature is motivated by already having a lower bound on the contribution of other types of customers within Constraint (V).} Constraints (I) and (II) respectively enforce the visibility and cardinality requirements, similar to the original formulation \ref{eq:SPVCd}. Constraints (III) and (IV) guarantee that for any customer $t\in [T]$ and any product class $q\in \{0,1,\ldots,Q\}$, the number of products in $D_q$ displayed to customer $t$ is consistent with the guessed packing pattern. Constraint (V) ensures that the lower bound $V_t$ on $V(S_t^*)$ is met for every customer $t\in [T]$. Constraint (VI) stipulates that for every light customer, the weight of her assortment is at most $\eps$. Finally, Constraint (VII) specifies that the decision variables $x_{it}$ are binary.  \black{It is worth mentioning that while some guesses from our procedure may result in infeasible linear programs, their infeasibility can be easily checked prior to proceeding with the algorithm, and they can be discarded as they are encountered. This does not affect the subsequent analysis, which relies on the guarantee that at least one guess --- specifically, the one corresponding to an optimal solution of \ref{eq:SPVCd} --- is optimal. As such, enumerating all the guesses will necessarily include this optimal guess. Finally, since our algorithm selects the best encountered guess, it is guaranteed to be near-optimal}.

Let $\mathbf{x^*}$ be an optimal solution to the linear relaxation of \ref{eq:ILP}, referred to as LP, and obtained by replacing the integrality constraint $x_{it}\in \{0,1\}$ by the relaxed constraint $x_{it} \in [0,1]$. The following lemma relates $\mathbf{x^*}$ to the optimal solution $S_1^*,\ldots, S_T^*$ of \ref{eq:SPVCd}. The proof of this result is included in Appendix \ref{apx:LPopt}.
\begin{lemma}\label{lem:LPopt}
    $
        \sum_{t=1}^T\frac{\sum_{i\in \Nc} v_ix^*_{it}}{1+\sum_{i\in \Nc} v_ix^*_{it}} \geq (1-\eps)\cdot\sum_{t=1}^T\frac{V(S_t^*)}{1+V(S_t^*)}.
    $
\end{lemma}

\subsection{Dependent Rounding}\label{subsec:GKPS}
Our next step is to round $\mathbf{x^*}$ in order to obtain {an approximate solution to} \ref{eq:SPVCd}.
To this end, we briefly review the dependent rounding scheme of \cite{gandhi2006dependent}, which will be employed to round our fractional solution into an integral one. Given a bipartite graph $(U,V,E)$ and values $x_{ij}\in [0,1]$ for every edge $(i,j)\in E$, \cite{gandhi2006dependent} present a polynomial-time \black{algorithm that generates} a collection of Bernoulli random variables $\{X_{ij}\}_{(i,j)\in E}$ satisfying the next three properties:
\setenumerate[1]{label={(P\arabic*)}}
\begin{enumerate}
    \item\label{it:P1} {\bf Marginal distributions:} $\E[X_{ij}] = x_{ij}$ for every edge $(i,j)$.
    \item\label{it:P2} {\bf Degree-preservation:} The fractional degree of every vertex is rounded to its floor or to its ceiling, almost surely. Namely, \black{for all} $i\in U\cup V$, \black{letting $d_i = \sum_{j\in \delta(i)}x_{ij}$, we have $\sum_{j\in \delta(i)}X_{ij}\in\{\lfloor d_i\rfloor, \lceil d_i\rceil\}$. Here, $\delta(i)$ is the set of neighbors of vertex $i$.}
    \item\label{it:P3} {\bf Negative correlation:} For every vertex $i\in U\cup V$ and subset $S\subseteq \delta(i)$ of its neighbors
    \begin{equation*}
        \P\left[\bigwedge_{j\in S}(X_{ij}=0)\right]\leq \prod_{j\in S}\P\left[X_{ij} = 0\right]\quad
        \text{and}\quad \quad\P\left[\bigwedge_{j\in S}(X_{ij}=1)\right]\leq \prod_{j\in S}\P\left[X_{i
        j} = 1\right].
    \end{equation*}
\end{enumerate}
As pointed out by \citet[Section~3]{gandhi2006dependent}, the negative correlation {(P3)} leads to Chernoff-type bounds, as stated in the following lemma.
\begin{lemma}   [{\cite{panconesi1997randomized}}]\label{lem:chernoff}
    Suppose that $X_1,\ldots,X_n$ satisfy property \ref{it:P3}. Then, for \black{any $\eps \in [0,1]$} and any vector $a=(a_1,\ldots,a_n)\in [0,1]^n$, letting $a(X) = \sum_{i=1}^na_iX_i$, we have$$
        \P[a(X)\leq (1-\eps)\cdot \E[a(X)]] \leq \exp\left(-\frac{\eps^2\cdot \E[a(X)]}{2}\right).
    $$
\end{lemma}

\subsection{Step 3:  The Rounding Procedure}\label{subsubsec:bipartite}
Next, we \black{devise} a framework \black{for employing} the \black{above-mentioned} dependent rounding scheme in order to derive \black{an approximate solution to \ref{eq:SPVCd}}.
Noticing that this rounding procedure takes \black{an edge-weighted} bipartite graph as its input, our first step is to interpret the optimal solution ${\bf x}^*$ of LP as a collection of edge-associated values in a carefully constructed bipartite graph.
Before presenting this graph, let us introduce the next two definitions. We say that a customer $t\in [T]$ is {\em bounded} when $k_{q,t} \neq \star$, for all $q=1,\ldots, Q$, noting that $q$ starts from $1$ in this definition rather than from $0$. Otherwise, we say that customer $t$ is {\em unbounded}. In other words, customer $t$ is bounded when in her optimal assortment $S_t^*$, she is offered at most $1/\eps^5$ products from any of the product classes $D_1,\ldots, D_Q$.
\paragraph{The bipartite graph.}
We start by defining the vertex sets $U$ and $V$. The set $U$ is simply the set $\Nc$ of all products. To describe the set of vertices $V$, for every unbounded customer $t$, we introduce a vertex $t$. In addition, for every bounded customer $t$, we introduce $Q+1$ vertices, specifically, one vertex $(t,q)$ for each class of products $q=0,\ldots, Q$. Therefore, $$V = \{t\,:\, t\text{ unbounded}\} \cup \left\{(t,q)\,\colon \, t \text{ bounded}, q =0,\ldots,Q\right\}.$$
Let us now describe the edge set $E$ of this graph, as well as the weight of each edge $e\in E$. \black{To this end, every} vertex $i\in U$ is joined by an edge to every vertex $t$ of every unbounded customer $t$; the edge $(i,t)$ is associated with the value $x^*_{it}$. Next, suppose that $D_q$ is the class containing product $i$. Then $i$ is joined by an edge to every vertex $(t,q)$ for \black{every bounded customer $t$}; the edge $(i, (t,q))$ is associated with the value $x^*_{it}$. \black{It is worth noting that for every bounded customer $t$ and every product $i\in \Nc$, there exists a unique edge $(i, (t,q))\in E$ in the above-constructed bipartite graph, where $q$ corresponds to the index of the class $D_q$ containing product $i$. As such, there is a one-to-one correspondence between all edges of the form $(i, (t,q))$ in our bipartite graph and the set of pairs $\{(i,t):i\in \Nc\text{ and }t\text{ bounded}\}$. This allows us to label each edge connecting a vertex $i\in U$ to a vertex of the form $(t,q)\in V$ simply as edge $(i,t)$, without any risk of ambiguity. Adopting this simplified notation, the weight of every edge $(i,t)$ in our bipartite graph is $x_{it}^*$, for every product $i\in \Nc$ and every customer $t\in [T]$.
}
\paragraph{{\color{black} Rounding procedure.}} {\black{With respect to this bipartite graph,}} we can now apply the dependent rounding procedure of \cite{gandhi2006dependent} to derive a collection of random variables $(X_{it}\,\colon\, i\in \Nc, t\in[T])$ representing rounded values. Finally, we return the associated sequence $S_1,\ldots, S_T$ of random assortments, where $S_t= \{i \in \Nc \,\colon\, X_{it}=1\}$ for each \black{customer} $t\in [T]$.
\paragraph{Running time.}Since our algorithm is based on exhaustive enumeration, the steps presented in Sections \ref{subsec:linearsol} and \ref{subsubsec:bipartite} are applied to each possible guess considered in Section \ref{subsec:linearsol}. Solving LP takes $O((nT)^{O(1)})$ using standard linear programming algorithms. Constructing the bipartite graph takes $O(nTQ)$ time, and implementing the dependent rounding scheme requires $O(|E|\cdot(|L|+|R|)) = O(n^2T^2Q)$ time \citep[Theorem~2.3]{gandhi2006dependent}. Therefore, processing each guess takes $O((nT)^{O(1)}Q) = O((nT)^{O(1)}\cdot\frac{1}{\eps}\log{\frac{1}{\eps}})$ time, since $Q = O(\frac{1}{\eps}\log\frac{1}{\eps})$. Finally, as explained in Section \ref{subsec:linearsol}, there are $O(T^{2^{O(\frac{1}{\eps}\log^2\frac{1}{\eps})}})$ possible guesses, \black{meaning that} the overall running time of our algorithm is
$$
O\left(T^{2^{O(\frac{1}{\eps}\log^2\frac{1}{\eps})}} \cdot (nT)^{O(1)}\cdot\frac{1}{\eps}\log{\frac{1}{\eps}}\right) = O\left(n^{O(1)}T^{2^{O(\frac{1}{\eps}\log^2\frac{1}{\eps})}}\right),
$$
which is polynomial in the input size for any fixed $\eps$.

\subsection{Analysis}\label{subsec:analysis}
In this section, we prove that the sequence of random assortments $(S_1,\ldots,S_T)$ is near-optimal for \ref{eq:SPVCd}, \black{as formally} stated in the following theorem.

\begin{theorem}\label{thm:feasibleoptimal}
    The sequence of random assortments $S_1,\ldots, S_T$ is feasible for \ref{eq:SPVCd}, with an expected \black{objective value of}
    \begin{equation*}
        \E\left[\sum_{t=1}^T \frac{V(S_t)}{1+V(S_t)} \right]\geq (1-3\eps) \cdot \widehat{\opt}.
    \end{equation*}
\end{theorem}

    In what follows, we start by showing that the sequence $S_1, \ldots,S_T$ is always feasible for \ref{eq:SPVCd}. Subsequently, we show that the expected \black{objective value} of this sequence is within factor $1-3\eps$ of optimal. We remind the reader that $S_1^*,\ldots, S_T^*$ is an optimal sequence of assortments for \ref{eq:SPVCd}.
    
    \paragraph{Feasibility.} First, to verify that the visibility constraint of every product $i\in \Nc$ \black{is satisfied}, \black{note that} by constraint (I) of LP, we have $\sum_{t=1}^Tx_{it}^* \geq \ell_i$. Noticing that $\sum_{t=1}^Tx_{it}^*$ is the fractional degree of vertex $i$ in the bipartite graph, by property \ref{it:P2}, we have $\sum_{t=1}^TX_{it} \geq \lfloor\ell_i\rfloor = \ell_i$.
    
    We now show that our cardinality constraint is always satisfied as well, meaning that each customer is offered at most $k$ products. First, for every unbounded customer $t$, the cardinality of $S_t$ is simply given by the degree of vertex $t$. Since $\sum_{i=1}^nx_{it}^* \leq k$, by property \ref{it:P2}, we have $|S_t| = \sum_{i\in \Nc}X_{it} \leq k$. For a bounded customer $t$, for every class $q=1,\ldots,Q$, the fractional degree of vertex $(t,q)$ is $\sum_{i\in D_q}
    x_{it}^*=k_{q,t}$, according to contraint (III) of \ref{eq:ILP}. Since $k_{q,t}$ is an integer, the degree of vertex $(t,q)$ remains unchanged with respect to $\bf X$, according to property \ref{it:P2}. Therefore, we almost surely have\begin{align}
        \sum_{i\in \Nc} X_{it}&= \sum_{i\in D_0}X_{it} + \sum_{q=1}^Q\sum_{i\in D_q}X_{it} \notag\\ &\leq \left\lceil\sum_{i\in D_0}x_{it}^*\right\rceil+ \sum_{q=1}^Q k_{q,t} \label{eq:5}\\ &= \left\lceil\sum_{i\in D_0}x_{it}^*+ \sum_{q=1}^Q k_{q,t} \right\rceil\label{eq:6}\\ &= \left\lceil\sum_{i\in \Nc}x_{it}^* \right\rceil \notag\\
        &\leq \lceil k\rceil \label{eq:7}\\
        &=k.\notag
    \end{align}
    Here, \black{Equation~\eqref{eq:5}} follows from Property \ref{it:P2}. \black{Equation~\eqref{eq:6} holds since} the second summand $\sum_{q=1}^Q k_{q,t}$ is an integer. \black{Equation~\eqref{eq:7}} follows from constraint (II) of LP. This proves the cardinality constraint and concludes the feasibility proof. 
    \paragraph{Near-optimality.} We \black{proceed by showing} that the expected contribution of every customer $t\in [T]$ to the objective function, i.e., $\E[V(S_t)/(1+V(S_t))]$ is within factor $1-2\eps$ of her contribution in the optimal solution $\mathbf x^*$. Formally, we have the following lemma, whose proof is provided in Appendix \ref{apx:nearoptimal}.
    \begin{lemma}\label{lem:nearoptimal}
        For every customer $t\in [T]$, we have$$
            \E\left[\frac{V(S_t)}{1+V(S_t)}\right] \geq (1-2\eps)\cdot\frac{\sum_{i\in \Nc}v_ix_{it}^*}{1+\sum_{i\in \Nc}v_ix_{it}^*}.
        $$
    \end{lemma}
We can now conclude the proof of Theorem \ref{thm:feasibleoptimal}, by observing that
     \begin{align}\E\left[\sum_{t =1}^T\frac{V(S_t)}{1+V(S_t)}\right]&\geq (1-2\eps)\cdot\sum_{t =1}^T\frac{\sum_{i\in \Nc}v_ix_{it}^*}{1+\sum_{i\in \Nc}v_ix_{it}^*} \label{eq:1024}\\
     &\geq (1-2\eps)\cdot(1-\eps)\cdot \sum_{t=1}^T\frac{V(S_t^*)}{1+V(S_t^*)} \label{eq:2048}\\&\geq (1-3\eps)\cdot \sum_{t=1}^T\frac{V(S_t^*)}{1+V(S_t^*)}\notag\\
    &=(1-3\eps)\cdot \widehat{\opt}\notag,
\end{align}
where Inequality \eqref{eq:1024} is an application of Lemma \ref{lem:nearoptimal} and Inequality \eqref{eq:2048} follows from Lemma \ref{lem:LPopt}.



\section{Price of Visibility} \label{sec:price}
In this section, we investigate the impact of visibility constraints on the total expected revenue, comparing this model to the unconstrained setting where there are no visibility constraints. In Section \ref{subsection:loss}, we quantify the maximum-possible loss resulting from enforcing visibility constraints. In Section \ref{subsection:share}, we introduce a novel method to distribute the loss incurred among different products in proportion to their share of the overall loss.

\subsection{The Loss Due to Visibility Constraints}
\label{subsection:loss}
Consider the unconstrained assortment optimization problem \ref{Unconstrained problem}, which was defined in Section \ref{subsectin:lp} as
\begin{equation}
\begin{aligned}
\max_{S \subseteq \mathcal{N}} \quad  R(S),
\end{aligned}
\tag{\sf{AP}}
\end{equation}
and let $S^*$ be its optimal solution. As mentioned in Section \ref{sec:intro}, this basic version of the MNL model has a price-ordered optimal assortment (\cite{talluri2004revenue}). \black{We refer the reader to Appendix \ref{apx1} for further discussion of assortment optimization under MNL}.
In the absence of visibility constraints, it is clearly optimal to offer the assortment $S^*$ to all customers. Consequently, the total expected revenue in the unconstrained setting is simply $T\cdot  R(S^*)$.
In the following example, we show that enforcing visibility constraints can lead to an arbitrarily \black{large revenue gap in comparison with} the unconstrained setting.

\vspace{2mm}
\noindent{\bf Example.} Consider an \ref{APV} instance with two products and $T$ customers. The product prices are $p_1=1$ and $p_2=0$, their preference weights are $v_1=1$ and $v_2 = M$, and the visibility constraints are given by $\ell_1=0$ and $\ell_2=T$, where $M$ is a large constant. As mentioned earlier, \black{without visilibity constraints}, it \black{suffices} to evaluate the expected revenue of all price-ordered assortments and choose the one with the highest revenue to obtain an optimal assortment for \ref{Unconstrained problem}. Since 
$R(\{1\})=  \frac{1}{2} $, $R(\{1,2\})=  \frac{1}{2+M} < R(\{1\})$, and $R(\emptyset) = 0$, the optimal assortment is $S^*=\{1 \}$ with $R(S^*)=\frac{1}{2}$.

On the other hand, let $(S_1, S_2, \ldots, S_T)$ be a feasible solution for \ref{APV}. Since $\ell_2=T$, we have to include product $2$ in each of these assortments. Moreover, adding product $1$ to any assortment only increases our revenue since $R(\{1,2\})=  \frac{1}{2+M} > R(\{2\})=0$.
Therefore, it is optimal to offer product $1$ and $2$ in every assortment, and $S_t^*=\{1,2\}$ for all $t =1,\ldots,T$. 
As such, the ratio between the optimal expected revenue collected from $T$ customers in an unconstrained setting and the one collected when enforcing visibility constraints is
$$\frac{T \cdot R(S^*)}{\sum_{t=1}^T R(S_t^*)} = \frac{T \cdot  \frac{1}{2}   }{T \cdot\frac{1}{2+M}}= \frac{M}{2}+1,$$
which tends to infinity as $M$ increases.
In this example, we see that enforcing products with very low price and high weight can drive the expected revenue \black{to arbitrarily low values}.


\subsection{Sharing the Expected Loss}
\label{subsection:share}


In this section, we explore a scenario of a platform, where each product within our universe is associated with a specific vendor. These vendors can impose visibility constraints on their products, established through service-level agreements or product sponsorships. As observed in Section \ref{subsection:loss}, enforcing such constraints may result in decreasing the platform's revenue. To address this issue, the platform wishes to implement a fee structure based on the vendors' contributions to the overall revenue loss. 
Letting $S^*$ and $(S_1^*, S_2^*,\ldots, S_T^*)$ be optimal solutions to \ref{Unconstrained problem} and \ref{APV}, we denote the revenue loss due to visibility constraints by 
\begin{equation}
    \label{Delta loss}
    \Delta \coloneqq T \cdot R(S^*) - \sum_{t=1}^T R(S_t^*).
\end{equation}

\noindent
{\bf A naive approach.} One approach for allocating $\Delta$ is solely based on the parameters $\ell_i$, which represent the lower bounds on the products' imposed visibility. \black{Toward this objective, suppose that} the proportion assigned to the vendor of product $i$ would be determined by \black{$\ell_i/\sum_{j=1}^n \ell_j$}. This distribution is generally not equitable, in the sense that  we should not impose any charges on a product that already belongs to the optimal set $S^*$, even when $\ell_i > 0$, as each such product is already profitable for the platform, and would be offered to customers regardless of its visibility requirement. Moreover, this allocation rule fails to consider the impact of each product on the overall loss. For example, when a product has an exceptionally high preference weight $v_i$ but a significantly lower price $p_i$, while other products have higher prices and lower preference weights, enforcing visibility for the first product would drive our revenue down, whereas the others would have a lesser impact. In this scenario, the former product should in principle be responsible for covering almost the entire revenue loss.

\noindent
{\bf Our proposed approach.}
Let us first observe that \black{$R(S_t^*) = \sum_{i \in S_t^*} p_i v_i/(1 + \sum_{i \in S_t^*} v_i)$}, implying that for every customer $t\in [T]$
$$R(S_t^*) = \sum_{i \in S_t^*}  (p_i - R(S_t^*))\cdot v_i.$$
This decomposition tells us which products are driving our revenue down (those with $p_i < R(S_t^*)$) and which products are acting in the opposite direction ($p_i \geq R(S_t^*)$). It also shows that the contribution of  product $i$ to the expected revenue is proportional to the difference between the price of product $i$ and the actual revenue $R(S_t^*)$, as well as proportional to the preference weight $v_i$. We can now rewrite the total revenue of the assortments $S_1^*, S_2^*, \ldots, S_T^*$ as 
$$\sum_{t=1}^T R(S_t^*) = \sum_{t=1}^T \sum_{i \in S_t^*}  (p_i - R(S_t^*)) \cdot v_i = \sum_{i=1}^n \sum_{t=1}^T \mathbbm{1}(i \in S_t^*) \cdot (p_i - R(S_t^*))\cdot v_i.$$

Therefore, we view the collective contribution of each product $i \in \mathcal{N}$ to the total revenue over all customers as $$\cont{i}\coloneqq\sum_{t=1}^T \mathbbm{1}(i \in S_t^*) \cdot (p_i - R(S_t^*)) \cdot v_i.$$


\noindent
{\bf Pricing the expected loss.}
 For each product $i \in \mathcal{N}$, we propose to charge its vendor the fraction of the loss corresponding to the negative contribution of this product, divided by the sum of all negative contributions, namely,

\begin{equation} \label{pricing formula}
  \Gamma_i \coloneqq \frac{\cont{i}^-}{\sum_{j=1}^n \cont{j}^-} \cdot \Delta,
\end{equation}
where $x^- = -\min(x, 0)$ is the negative part of \black{$x$}, and $\Delta$ is the total loss in revenue due to enforcing visibility constraints, \black{given by Equation~\eqref{Delta loss}}. Next, we discuss three important properties of this strategy:
\paragraph{(a) Fair distribution.} First, it is worth noting that the revenue loss $\Delta$ is exactly shared between the products whose contribution is negative. Indeed, on the one hand $\sum_{i \in \mathcal{N}} \Gamma_i = \Delta$, and on the other hand  $\Gamma_i > 0$ if and only if $C_i^->0$, i.e., the product's contribution to the revenue is negative. Moreover, the fee $\Gamma_i$ for each product takes into account its actual impact. The first observation is that all products in $S^*$ have nonnegative contributions, and their vendors are consequently exempt from a fee, as expected. The second observation is that for any product with a negative contribution, its revenue impact is amplified by a lower price and a higher preference weight. It turns out that our policy \black{accounts for this observation}, as the lower the price of such a product, the greater revenue loss it causes, and thus the higher the fee it imposes on the vendor. Similarly, the greater the preference weight of an unprofitable product (i.e., one with $p_i< R(S^*)$) the greater its fee. In the following, we say that a product $i\in \Nc$ is unprofitable with respect to an assortment $S\subseteq \Nc$ if $p_i< R(S)$. The third observation is that even if a product $i\in \Nc$ is unprofitable with respect to the optimal unconstrained assortment $S^*$, it can still have a nonnegative contribution, in which case $\Gamma_i =0$, concurring with a product whose overall impact is nonnegative. In fact, the same holds even if a product is unprofitable with respect to some assortment $S_t^*$ for $t\in [T]$, as its overall contribution may potentially be nonnegative. The final observation is that this pricing strategy guarantees a fair treatment of identical products. Indeed, if two distinct products $i$ and $j$ have identical prices and preference weights, then imposing similar visibility constraints implies a similar fee for the two vendors.
\paragraph{(b) Monotonicity.}                            
For every product $i$, the fee $\Gamma_i$ is nondecreasing as a function of $\ell_i$. In other words, if a vendor wishes to display her product more often to customers, the resulting fee becomes higher. \black{Formally, let us fix a product $i\in \Nc$, and consider the instance with $T$ customer, $\Tilde \ell_i = \ell_i+1$ and $\Tilde \ell_j = \ell_j$ for $j\neq i$, which corresponds to increasing the visibility requirement $\ell_i$ of product $i$ by $1$, while the requirement for other products remains unchanged. Let $\Tilde \Gamma_i$ be the fee imposed on product $i$ in the latter instance. Then, we have the following claim, whose proof is included in Appendix~\ref{apx:increasefee}.
\begin{claim}\label{cl:increasefee}
    $\Tilde \Gamma_i \geq \Gamma_i$.
\end{claim}
}
\paragraph{(c) Computational tractability.} The fees charged are easy to compute: Each $\Gamma_i$ can be computed in ${O}(nT)$ time, since it only requires solving \ref{APV}. Additionally, consider the situation where a vendor is interested in knowing the fee that she will be paying in order to increase the visibility of her product by one unit, corresponding to increasing $\ell_i$ to $\ell_i + 1$, without changing any other ${\ell}_j$.
In this case, we only have to recompute $S_{\ell_i + 1}^*$, since all others assortments $S_t^*$ remain the same. Therefore, for a fixed product $i$, we can efficiently compute the fee $\Gamma_i$ for all $\ell_i \in [T]\cup\{0\}$.
 
   
\black{\paragraph{Additional discussions.} Here, we discuss the sensitivity of the fee with respect to problem primitives, mainly the preference weights. Since the MNL choice model is a stylized model for customer choice behavior, an estimation step is applied prior to making assortment decisions, in order to determine estimates for the model's preference weights using methods such as maximum likelihood estimation or Bayesian methods; see Chapter~8 in \cite{train2009discrete} for more detail on numerical estimation procedures in choice modelling. Such estimates are by assumption, approximations of some underlying real parameters, and are inherently noisy. This noise in estimation can have an impact on the fees given to customers. Indeed, let us take a simple illustrating argument. Consider an instance with a single customer, i.e., $T=1$, and two products $1$ and $2$, with their respective prices $p_1 = 2$ and $p_2 = 1$, and their visibility requirements $\ell_1 = 0$ and $\ell_2 = 1$. According to some prior estimation procedure, the preference weights of both products are $v_1 = v_2 = 1$. Abstracting about the technical details, applying our approach to the above instance would show that $\{1,2\}$ is an optimal solution, and that the expected loss is $\Delta = 0$. Consequently, both products enjoy a fee of $0$. Now assume that due to estimation error, the true underlying parameter is $\Tilde v_1 = 1+\eps$ for some estimation error $\eps>0$, and the parameter $v_2$ was correctly estimated. Therefore, the true expected loss is $$\Tilde \Delta = \frac{\eps}{(2+\eps)(3+\eps)},$$
which is entirely incured by product $2$, that is $\Gamma_2 = \Tilde \Delta$. This shows the existence of instances where $\eps$-bounded variations on the preference weights can turn a null fee into a positive one.

It is noteworthy that while our approach outlines an explicit cost-sharing scheme, it is much richer as it quantifies the so-called contribution $C_i$ of any product $i\in \Nc$. This allows platforms to design specific cost-sharing schemes, tailored to their estimation procedures, business needs and their bottom line. For example, recalling the above example, when the platform's estimation procedure allows for an $\eps$-bounded noise, the preference weights can for instance be randomly sampled from a distribution over confidence intervals for preference weights.
}

    \section{Concluding Remarks}
    \label{sec:conclusions}
    


\black{We conclude this paper by presenting a number of intriguing questions and promising avenues for future research.}

{\em Algorithms for \ref{APVC} with general prices}: While we designed an approximation scheme for \ref{APVC} with identical product prices, developing non-trivial algorithms for general prices appears to be a challenging research direction. Extending our algorithmic approach to \ref{APVC} with general prices will likely require new ideas, as directly applying our current techniques does not appear to yield a PTAS. Specifically, discretizing and guessing techniques similar to those of Section \ref{subsec:linearsol} would need additional refinement to effectively handle this broader setting.

{\em Generalization to arbitrary choice models}: Investigating assortment optimization with visibility constraints under alternative choice models, such as Nested Logit, Markov Chain, or ranking-based choice models, represents another interesting area for exploration. Our current algorithms are heavily reliant on structural properties of the MNL choice model. For instance, our approach for computing the expanded set leverages the concept of revenue-ordered assortments, which does not easily generalize to other choice models. Additionally, the supermodularity of the expanded revenue is based on the formulation of the MNL choice probabilities, making its applicability to a general choice model an intriguing question for future research.

{
\addcontentsline{toc}{section}{Bibliography}
\bibliographystyle{plainnat}
\bibliography{BIB-Visibility}
}

\appendix
\changelocaltocdepth{1}

    \section{MNL-based Revenue Maximization}
    \label{apx1}
For completeness, we present two well-known results in the context of unconstrained revenue maximization problem under the MNL model, defined as follows
\begin{equation}
\label{Unconstrained}\tag{\sf{AP}}
\max_{S \subseteq \mathcal{N}} \quad  R(S).
\end{equation} 
Lemma \ref{Revenue variations} provides necessary and sufficient conditions under which adding a product to a given assortment increases its expected revenue. Then, Lemma \ref{lem:nestedrevenue}, shows that the optimal assortment for \ref{Unconstrained} is price-ordered (see, e.g., \citet[Prop.~6]{talluri2004revenue}).

The next lemma shows that adding a product $j$ to any given assortment $S$ weakly increases its expected revenue if and only if the price of this product $p_j$ is at least $R(S)$, or equivalently, at least $R(S\cup\{j\})$.

\begin{lemma}\label{Revenue variations}
    For any assortment $S\subseteq {\cal N}$ and any product $j \in {\cal N} \setminus S$, the next three conditions are equivalent:
$$(i) \; R(S \cup \{j\}) \geq R(S),  \qquad (ii)  \; p_j \geq R(S),  \qquad (iii) \; p_j \geq R(S\cup \{j\}). $$
\end{lemma}

\begin{lemma}\label{lem:nestedrevenue}
    \label{Unconstrained solution}
    Let $R^* = \max_{S \subseteq \mathcal{N}} R(S)$. Then, the price-ordered assortment $S^* = \{i \in \mathcal{N}, p_i \geq R^* \}$ is optimal.
\end{lemma}

    
\begin{proof}[Proof of Lemma \ref{Revenue variations}]
    We show that $R(S\cup \{j\})$ is a convex combination of $R(S)$ and $p_j$. To this end, note that
    \begin{align*}
        R\left(S\cup \{j\}\right) &= \sum_{i\in S} p_i\cdot\phi\left(i,S\cup \{j\}\right) + p_j\cdot\phi\left(j,S\cup \{j\}\right)
     = \alpha \cdot R(S)  + (1-\alpha)\cdot p_j,
    \end{align*}
    where $\alpha = \frac{1+V(S)}{1+V(S\cup\{j\})}$. It is easy to verify that $\alpha \in (0,1]$,  which proves that $R(S\cup \{j\})$ is a convex combination of $R(S)$ and $p_j$. Consequently, $R(S\cup \{j\})$ belongs to the closed interval whose endpoints are $R(S)$ and $p_j$. In particular, we have $R(S\cup \{j\}) \geq R(S)$ if and only if $p_j \geq R(S)$, and if and only if $p_j \geq R(S\cup \{j\})$.
\end{proof}

\begin{proof}[Proof of Lemma \ref{lem:nestedrevenue}]
Let $\Tilde{S}$ be an optimal assortment of \ref{Unconstrained} with maximal cardinality. In the following, we show that $S^* = \Tilde S$.
We begin by arguing that $S^*\subseteq \Tilde S$. Consider some product $i\in S^*$, and assume by contradiction that $i\notin \Tilde S$. On the one hand, we know by optimality of $\Tilde S$ that $R(\Tilde S) = R^*$. On the other hand, $p_i\geq R^*$ by definition of $S^*$. Therefore, $p_i\geq R(\Tilde S)$ and we can apply \black{Lemma} \ref{Revenue variations}, which implies that $R(\Tilde{S}\cup \{i\})\geq R(\Tilde S)$. Hence, $\Tilde{S}\cup \{i\}$ is also an optimal assortment, contradicting the definition of $\Tilde{S}$ as an optimal assortment with maximal cardinality.
Now, to show that $\Tilde S\subseteq S^*$, suppose there exists a product $i\in \Tilde S$ that does not belong to $S^*$, i.e., $p_i < R^* = R(\Tilde S)$. By Lemma \ref{Revenue variations}, the latter inequality implies that $R(\Tilde S\setminus \{i\}) > R(\Tilde S)$, which contradicts the optimality of $\Tilde S$.
\end{proof}

\section{\black{Proof of Claim \ref{cl:computation}}}\label{apx:computation}
\black{Let $S_1, S_2\subseteq \Nc$ with $S_1\subseteq S_2$. Then, \begin{align*}
    R(S_2) &= \sum_{j\in S_2} \frac{p_j\cdot v_j}{1+V(S_2)} \\
    & = \sum_{j\in S_1} \frac{p_j\cdot v_j}{1+V(S_2)}  + \sum_{j\in S_2\setminus S_1} \frac{p_j\cdot v_j}{1+V(S_2)}\\
    & = R(S_1) \cdot \frac{1+V(S_1)}{1+V(S_2)} + \sum_{j\in S_2\setminus S_1} \frac{p_j\cdot v_j}{1+V(S_2)}
\end{align*}
Therefore,
\begin{align*}
    R(S_1) - R(S_2) &= R(S_1) \cdot \left(1 - \frac{1+V(S_1)}{1+V(S_2)}\right) - \sum_{j\in S_2\setminus S_1} \frac{p_j\cdot v_j}{1+V(S_2)}\\
    & = R(S_1) \cdot  \frac{\sum_{j\in S_2\setminus S_1}v_j}{1+V(S_2)} - \sum_{j\in S_2\setminus S_1} \frac{p_j\cdot v_j}{1+V(S_2)}\\
    & = \frac{1}{1+V(S_2)}\cdot \sum_{j\in S_2\setminus S_1}\left(R(S_1)  - p_j\right)\cdot v_j
\end{align*}
}

\section{Additional Proofs from Section \ref{sec:APVC}}\label{apx:APVC}
\subsection{Proof for Theorem \ref{NP-hardness}}\label{apxhard}
In order to prove that \ref{APVC} is strongly NP-hard, we reduce the $3$-\texttt{PARTITION} problem to \ref{APVC}. \black{In this setting}, the input consists of $3T$ positive integers ${\cal A} = \{a_1,\ldots, a_{3T}\}$ with $\sum_{i=1}^{3T}a_i=BT$, for some $B>0$. We are asked to determine whether $\cal A$ can be partitioned into $T$ triplets $A_1,\ldots, A_T$, such that the sum of each triplet is exactly $B$. This problem is known to be strongly NP-hard (see, e.g., \citet[Theorem~3.5]{garey1975complexity})
\paragraph{Instance construction.} Given a $3$-\texttt{PARTITION} instance of the form described above, we construct the following \ref{APVC} instance. We consider \black{$T$ customers and} $3T$ products, comprising the universe $\Nc = \{1,\ldots,3T\}$.  Each product $i\in \Nc$, has a preference weight $v_i = a_i$, a price $p_i = 1$, and its visibility requirement is $\ell_i = 1$. In particular, note that $V(\Nc) = \sum_{i=1}^{3T} a_i= BT$. Finally, the upper bound on the cardinality of any offered assortment is $k=3$.
\paragraph{Analysis.} An important observation is that any feasible solution for this instance offers each product to exactly one customer. In other words, if $S_1,\ldots ,S_T$ is a feasible solution, then for any product $i\in \Nc$, there exists a unique customer $t\in [T]$ for which $i\in S_t$. Indeed, on the one hand, the visibility constraints of $\ell_1=\ell_2= \ldots= \ell_{3T} =1$ impose that each product will be shown at least once, and on the other hand, since each assortment can contain at most $k=3$ products and there are $3T$ products in our universe, each product can be offered at most once. As a consequence, for any feasible solution $S_1,\ldots, S_T$, we have $\sum_{t=1}^TV(S_t) = V(\Nc) = BT$.

Now, let us denote by $S^* = (S^*_1,\ldots,S^*_T)$ an optimal collection of assortments for our \ref{APVC} instance. The following claim shows that a valid partition exists for the original $3$-\texttt{PARTITION} instance if and only if the expected revenue of $S^*$ exceeds a certain threshold. A close inspection of this result reveals that it directly implies the strong NP-hardness of \ref{APVC}, even with equal prices and integer-valued preference-weights.
\begin{lemma}\label{lem:NPhard}
    ${\cal A} = \{a_1,\ldots, a_{3T}\}$ is a YES-instance of $3$-\texttt{PARTITION} if and only if $\sum_{t=1}^TR(S_t^*) \geq T\cdot \frac{B}{1+B}.$
\end{lemma}
\paragraph{First direction: YES-instance $\boldsymbol{\implies \sum_{t=1}^TR(S_t^*) \geq T\cdot \frac{B}{1+B}}$.}
Suppose that $A_1,\ldots, A_T$ is a valid partition, and let $S_t = \{i\,\colon\,a_i\in A_t\}$. Then, $$
    \sum_{t=1}^TR(S_t^*) \geq \sum_{t=1}^TR(S_t)  = T\cdot \frac{B}{1+B}. 
$$
\paragraph{Second direction: NO-instance $\boldsymbol{\implies \sum_{t=1}^TR(S_t^*) \leq T\cdot \frac{B}{1+B} - \frac{2}{B(B+1)(B+2)}}$.}
Consider the following integer program:
    \begin{equation}\tag{IP}\label{eq:integerprogram}
        \begin{aligned}
         \max_{\mathbf{b} \in \N^T}  & \; \;      \sum_{t=1}^T  \frac{b_t}{1+b_t}   \\  
          {\rm s.t.} \;\;   & \;\;  \sum_{t=1}^Tb_t = BT.
        \end{aligned}
    \end{equation}
For any feasible solution $(b_1,\ldots,b_T)$, we have{\black{\begin{align}
    \sum_{t=1}^T \frac{b_t}{1+b_t}& = T\cdot \sum_{t=1}^T \frac1T\cdot \frac{b_t}{1+b_t} \notag\\ & \leq T\cdot \frac{\frac{1}{T}\cdot\sum_{t=1}^T{b_t}}{1+ \frac{1}{T}\cdot\sum_{t=1}^T{b_t}} \label{eq:15}\\ &= T\cdot \frac{B}{1+B}.\label{eq:bound}
\end{align}}}\\
\black{Inequality~\eqref{eq:15}} follows from Jensen's inequality and the concavity of $x\mapsto x/(1+x)$, and \black{Equality~\eqref{eq:bound}} is a consequence of $\sum_{t=1}^Tb_t = BT$. 
In other words, $\opt({\rm \ref{eq:integerprogram}}) \leq T\cdot B/(1+B)$.
    
Now consider any feasible sequence of assortments $(S_1,\ldots, S_T)$. Since ${\cal A}$ is a NO-instance, there exists \black{at least one pair of customers} $u,v\in [T]$ for which $V(S_u) \geq B+1$ and $V(S_v)\leq B-1$. We assume without loss of generality that $u=1$ and $v=2$. Consider the feasible solution $\mathbf b$ for \ref{eq:integerprogram} defined as follows: $b_1 = V(S_1)-1$, $b_2 = V(S_2) +1$ and $b_t = V(S_t)$ for all $t = 3,\ldots,T$.
\black{Then, to obtain an upper bound on $\sum_{t=1}^TR(S_t)$, we observe that} \begin{align}
    T\cdot \frac{B}{1+B} - \sum_{t=1}^TR(S_t) &\geq \sum_{t=1}^T\frac{b_t}{1+b_t} - \sum_{t=1}^T\frac{V(S_t)}{1+V(S_t)}\label{eq:one}\\
    & = \left(\frac{V(S_1)-1}{V(S_1)} - \frac{V(S_1)}{1+V(S_1)}\right) + \left(\frac{V(S_2)+1}{V(S_2)+2} - \frac{V(S_2)}{1+V(S_2)}\right)\label{eq:two}\\
    & = \frac{1}{(1+V(S_2))(2+V(S_2))} - \frac{1}{V(S_1)(1+V(S_1))}\notag\\
    & \geq \frac{1}{B(B+1)} - \frac{1}{(B+1)(B+2)}\label{eq:three}\\
    & = \frac{2}{B(B+1)(B+2)}.\notag
\end{align}
Here, Equation \eqref{eq:one} \black{holds since $\opt({\rm \ref{eq:integerprogram}}) \leq T\cdot \frac{B}{1+B}$, as shown above}. In Equation \eqref{eq:two}, we simply substitute the entries of $\mathbf b$. In Equation \eqref{eq:three}, we use the fact that $V(S_1)\geq B+1$ and $V(S_2) \leq B-1$.
\paragraph{Non-existence of an FPTAS.}
Suppose by contradiction that an FPTAS for \ref{APVC} exists. Letting $\eps = \frac{1}{TB^2(B+2)}$, since this quantity is pseudo-polynomial in the input size, we can obtain a $(1-\eps)$-approximation $(S_1,\ldots, S_T)$ in pseudo-polynomial time.
As explained above, we have $\sum_{t=1}^TR(S_t) > T\cdot \frac{B}{1+B}-\frac{2}{B(B+1)(B+2)}$ if and only if the $3$-\texttt{PARTITION} instance ${\cal A} = \{a_1,\ldots, a_{3T}\}$ is a YES-instance. Since $(S_1,\ldots, S_T)$ is obtained in pseudo-polynomial time, this yields a pseudo-polynomial time to solve the $3$-\texttt{PARTITION} problem, which contradicts its strong NP-hardness.

\subsection{Proof of Lemma \ref{lem:epsobjective}}\label{apx:epsobjective}
    First, since \black{$\nota{v}_i \leq v_i$ for every product $i\in\Nc$}, $\nota{V}(\nota{S}_t)\leq V(S_t)$ for all $t\in [T]$. Consequently, by the monotonicity of $x\mapsto x/(1+x)$, we have $\nota{V}(\nota{S}_t)/(1+\nota{V}(\nota{S}_t))\leq V(S_t)/(1+V(S_t))$. By summing this inequality over all customers $t\in [T]$, we obtain $\widehat{\obj}(\nota{S}_1,\ldots, \nota{S}_T)\leq \obj(S_1,\ldots, S_T)$. Now, let us prove the second inequality, stating that $V(S_t)/(1+V(S_t)) \leq (1+\eps)\cdot \nota{V}(\nota{S}_t)/(1+\nota{V}(\nota{S}_t))$, for     every $t \in [T]$.\begin{itemize}
        \item If there exists a product $i\in S_t$ with $v_i\geq 1/\eps$, then $\nota{v}_i = 1/\eps$, and in turn, $\nota{V}(\nota{S}_t)\geq 1/\eps$. Therefore, \begin{equation*}
            \frac{V(S_t)}{1+V(S_t)} \leq 1 = (1+\eps) \cdot \frac{1/\eps}{1+1/\eps} \leq (1+\eps)\cdot\frac{\nota{V}(\nota{S}_t)}{1+\nota{V}(\nota{S}_t)}
        \end{equation*}
        \item If $v_i< 1/\eps$ for all $i\in S_t$: Note that Equation \eqref{eq:weightbound} trivially holds when $v_i<\eps$, since $\nota{v}_i = v_i$ by definition. Therefore, this equation holds for every $i\in S_t$. Hence, $v_i/(1+v_i) \leq (1+\eps) \cdot \nota{v}_i/(1+\nota{v}_i)$, and our desired inequality follows by summing over $t=1,\ldots,T$.
    \end{itemize}

\subsection{Proof of Lemma \ref{lem:LPopt}}\label{apx:LPopt}
    In this proof, we establish separate inequalities for light, medium, and heavy customers. At the very end, combining these inequalities will lead to the result stated in Lemma \ref{lem:LPopt}.
    
    \paragraph{Heavy customers.} When $t\in [T]$ is a heavy customer, by constraint (V) of LP, we have $\sum_{i\in \Nc} v_ix^*_{it} \geq \frac{1}{\eps}.$
        Therefore, \begin{equation}\label{eq:heavy}
            \frac{\sum_{i\in \Nc} v_ix^*_{it}}{1+\sum_{i\in \Nc} v_ix^*_{it}} \geq \frac{1}{1+\eps} \geq (1-\eps) \cdot \frac{V(S_t^*)}{1+V(S_t^*)}.
        \end{equation}
    
    \paragraph{Medium customers.} In this case, customer $t$ belongs to one of the classes $G_1,\ldots, G_L$, say $G_{\ell}$. As a result,
        \begin{equation*}
            \sum_{i\in \Nc} v_i x^*_{it} \geq V_t = \eps\cdot (1+\eps)^{\ell-1} \geq \frac{1}{1+\eps}\cdot V(S_t^*),
        \end{equation*}
        where the first inequality follows from constraint (V) of LP, and the second inequality holds since $V(S_t^*) \leq \eps \cdot (1+\eps)^{\ell}$ by definition of the class $G_{\ell}$.
        Therefore,\begin{align}
            \frac{\sum_{i\in \Nc} v_ix^*_{it}}{1+\sum_{i\in \Nc} v_ix^*_{it}} &\geq  \frac{1}{1+\eps} \cdot \frac{V(S_t^*)}{1+\frac{1}{1+\eps}\cdot V(S_t^*)}\label{eq:2050}\\
            & \geq \frac{1}{1+\eps} \cdot \frac{V(S_t^*)}{1+V(S_t^*)}\notag\\ 
            &\geq  (1-\eps)\cdot \frac{V(S_t^*)}{1+V(S_t^*)}\label{eq:medium}.
        \end{align}
    \paragraph{Light customers.} Let $\mathbf{\hat X}$ be the vector given by $\hat X_{it} = \mathbbm 1\{i \in S_t^*\}$, for every $t\in [T]$ and $i\in \Nc$. According to our construction, it is easy to verify that $\mathbf{\hat X}$ is a feasible solution to \ref{eq:ILP}, and therefore, also to its linear relaxation LP. Therefore, since $\mathbf{x^*}$ is an optimal solution to the latter program, we have \begin{equation}\label{eq:noinspi}
            \sum_{t\in G_{\rm light}} \sum_{i\in \Nc} v_ix^*_{it} \geq \sum_{t\in G_{\rm light}} \sum_{i\in \Nc} v_i\hat X_{it} =  \sum_{t\in G_{\rm light}} V(S_t^*).
        \end{equation}
        Also, by constraint (VI) of LP, we have $\sum_{i\in \Nc} v_i x^*_{it}\leq \eps$ for all $t\in G_{\rm light}$. Therefore, \begin{align}
            \sum_{t\in G_{\rm light}}\frac{\sum_{i\in \Nc} v_ix^*_{it}}{1+\sum_{i\in \Nc} v_ix^*_{it}}&\geq \frac{1}{1+\eps}\cdot \sum_{t\in G_{\rm light}}\sum_{i\in \Nc} v_ix^*_{it}\notag\\
            &\geq (1-\eps)\cdot\sum_{t\in G_{\rm light}} V(S_t^*)\notag\\
            & \geq (1-\eps)\cdot\sum_{t\in G_{\rm light}}\frac{V(S_t^*)}{1+V(S_t^*)},\label{eq:light}
        \end{align}
    where the second inequality follows from Equation \eqref{eq:noinspi}.

\subsection{Proof of Lemma \ref{lem:nearoptimal}}\label{apx:nearoptimal}
    \noindent {\em {Case 1}: Customer $t$ is unbounded.} Let $q_t \in \{1,\ldots,Q\}$ be a class of products for which $k_{q_t,t} = \star$; such a class indeed exists by definition of unbounded customers. Then, \black{\begin{align}
        \sum_{i\in D_{q_t}} v_i x_{it}^* &\geq \eps^5\cdot \sum_{i\in D_{q_t}} x_{it}^* \label{eq:8} \\ &\geq \eps^5\cdot \frac{1}{\eps^6} \label{eq:9}\\ & = \frac{1}{\eps}.\label{eq:number1}
    \end{align}
Inequality~\eqref{eq:8} holds since $v_i \geq \eps^5$ for every product $i\in D_{q_t}$, as  $q_t\neq 0$.} \black{Equation~\eqref{eq:9}} follows from constraint (IV) of LP.
Therefore, 
    \begin{align}
        \P\left[V(S_t) \leq  \frac{1-\eps}{\eps}\right] &= 
        \P\left[\sum_{i\in \Nc}v_i X_{it} \leq  \frac{1-\eps}{\eps}\right] \notag\\& \leq \P\left[\sum_{i\in \Nc}v_i X_{it} \leq  (1-\eps)\cdot \sum_{i\in D_{q_t}} v_i x_{it}^*\right]\label{eq:ineq1}\\
        & \leq \P\left[\sum_{i\in D_{q_t}}v_i X_{it} \leq  (1-\eps)\cdot \sum_{i\in D_{q_t}} v_i x_{it}^*\right]\notag\\
        & = \P\left[\sum_{i\in D_{q_t}}X_{it} \leq  (1-\eps)\cdot \sum_{i\in D_{q_t}}x_{it}^*\right]\label{eq:number2}\\
        & \leq \exp\left(-\frac{\eps^2\cdot \sum_{i\in D_{q_t}}x_{it}^*}{2}\right)\label{eq:ineq2}\\
        & \leq \exp\left(-\frac{1}{2\eps^4}\right)\label{eq:ineq3}\\
        &\leq \eps.\label{eq:unboundbound}
    \end{align}
    Here, Inequality \eqref{eq:ineq1} follows from Equation \eqref{eq:number1}. Equality \eqref{eq:number2} holds since all products within the same class ($D_{q_t}$) have the same preference weight. Inequality \eqref{eq:ineq2} is a direct application of Lemma \ref{lem:chernoff}. In Inequality \eqref{eq:ineq3}, we plug in $\sum_{i\in D_{q_t}} x_{it}^*> 1/\eps^6$, since $k_{q_t, t} = \star$, as established by constraint (IV) of LP.
    Therefore, \begin{align}
        \E\left[\frac{V(S_t)}{1+V(S_t)}\right] &\geq \P\left[V(S_t)> \frac{1-\eps}{\eps}\right] \cdot  \E\left[\frac{V(S_t)}{1+V(S_t)} \,\left|\,V(S_t)> \frac{1-\eps}{\eps}\right.\right]\notag\\
        &\geq (1-\eps)\cdot \frac{\frac{1-\eps}{\eps}}{1+\frac{1-\eps}{\eps}} \label{eq:10} \\ &=(1-\eps)^2 \\ &\geq (1-2\eps)\cdot \frac{\sum_{i\in \Nc}v_ix_{it}^*}{1+\sum_{i\in \Nc}v_ix_{it}^*},\notag
    \end{align}
    where \black{Equation~\eqref{eq:10}} follows from Equation \eqref{eq:unboundbound} and the monotonicity of $x\mapsto x/(1+x)$ on $[0,\infty)$.

    \noindent{\em {Case 2}: Customer $t$ is bounded and heavy, or customer $t$ is medium.} To address this case, let us recall that the random preference weight of the assortment $S_t$ offered to customer $t$ is\begin{equation*}
    V(S_t)=\sum_{i\in \Nc}v_iX_{it} = \underbrace{\sum_{i\in D_0}v_iX_{it}}_{W_t^{\sma}} + \underbrace{\sum_{q = 1}^{Q-1}\sum_{i\in   D_q}v_iX_{it}}_{W_t^{\lar}}.
\end{equation*}
    \black{We introduce the random variables $W_t^{\sma}$ and $W_t^{\lar}$, which represent the contributions to $V(S_t)$ from products in $D_0$ and in $\Nc\setminus D_0$, respectively, defined as: $
    W_t^{\sma} = \sum_{i\in D_0}v_iX_{it}$ and $W_t^{\lar} = \sum_{q = 1}^{Q-1}\sum_{i\in   D_q}v_iX_{it}.
    $}
    Similarly, let $w_t^{\sma} = \sum_{i\in D_0}v_ix^*_{it}$, and $w_t^{\lar} = \sum_{i\in \Nc\setminus D_0}v_ix^*_{it}$, be the expected values of these  random variables. \black{In the next claim, we show that the random variable $W_t^{\lar}$ is almost surely equal to its expected value. Its proof, whose deferred to Appendix~\ref{apx:nonlightcust}, is a direct application of the degree preservation property}.
\begin{claim}\label{cl:nonlightcust}
    With probability $1$, we have $W_t^{\lar} = w_t^\lar.$
\end{claim}
\begin{itemize}
    \item When $w_t^{\sma}\leq \eps \cdot\sum_{i\in \Nc} v_ix_{it}^*$, then intuitively, the contribution of the products in $D_0$ represents at most an $\eps$-fraction of the total contribution of all products to $\sum_{i\in \Nc}v_ix_{it}^*$, and can therefore be neglected. Formally, \black{with} probability $1$, we have
    \black{\begin{align}
        V(S_t)&\geq W_t^{\lar}\label{eq:11} \\  &= w_t^{\lar} \notag\\
         &=  \sum_{i\in \Nc}v_i x_{it}^* - w_t^\sma \notag \\ \label{eq:12}&\geq (1-\eps)\cdot \sum_{i\in \Nc}v_ix_{it}^*,
    \end{align}}
    where \black{Equation~\eqref{eq:11}} is a consequence of Claim \ref{cl:nonlightcust}, and \black{Equation~\eqref{eq:12}} follows from the case hypothesis.
    Therefore, by the monotonicity of $x\mapsto x/(1+x)$, we have\begin{equation}
        \frac{V(S_t)}{1+V(S_t)} \geq (1-\eps)\cdot \frac{\sum_{i\in \Nc}v_i x_{it}^*}{1+\sum_{i\in \Nc}v_i x_{it}^*}.
    \end{equation}
    \item When $w_t^{\sma}> \eps \cdot\sum_{i\in \Nc} v_ix_{it}^*$, we have
    \begin{align}
        \P\left[W_t^\sma \leq (1-\eps)\cdot w_t^\sma\right] & = \P\left[\sum_{i\in D_0}\frac{v_i}{\eps^5}X_{it} \leq (1-\eps)\cdot \sum_{i\in D_0}\frac{v_i}{\eps^5}x_{it}^*\right]\notag\\
        &\leq \exp\left(-\frac{ \sum_{i\in D_0}v_ix_{it}^*}{2\eps^3}\right)\label{eq:ineq4}\\
        &\leq \exp\left(-\frac{\sum_{i\in \Nc}v_ix_{it}^*}{2\eps^2}\right)\label{eq:ineq5}\\
        &\leq \exp\left(-\frac{1}{2\eps}\right)\label{eq:ineq6}\\
        &\leq \eps.\label{eq:boundbound}
    \end{align}
    Here, Inequality \eqref{eq:ineq4} follows from Lemma \ref{lem:chernoff}, since $v_i< \eps^5$ for all $i\in D_0$. Inequality \eqref{eq:ineq5} is a consequence of the case hypothesis. Inequality \eqref{eq:ineq6} follows from the fact that customer $t$ is not light, and hence $\sum_{i\in \Nc}v_ix_{it}^* \geq \eps$.
    Therefore, \begin{align}
        \E\left[\frac{V(S_t)}{1+V(S_t)}\right]
        &\geq \P\left[W_t^{\sma}> (1-\eps)\cdot w_t^{\sma}\right] \cdot  \E\left[ \left.\frac{V(S_t)}{1+V(S_t)}\,\right|\,W_t^{\sma}> (1-\eps)\cdot w_t^{\sma}\right]\notag\\
        &= \P\left[W_t^{\sma}> (1-\eps)\cdot w_t^{\sma}\right] \cdot  \E\left[ \left.\frac{w_t^\lar+ W_t^\sma}{1+w_t^\lar+ W_t^\sma}\,\right|\,W_t^{\sma}> (1-\eps)\cdot w_t^{\sma}\right]\label{eq:13}\\
        &\geq (1-\eps) \cdot \frac{w_t^\lar+ (1-\eps)\cdot w_t^s}{1+w_t^\lar+(1-\eps)\cdot w_t^s}\label{eq:14}\\
        &\geq (1-\eps)^2.\frac{\sum_{i\in \Nc}v_ix^*_{it}}{1+\sum_{i\in \Nc}v_ix^*_{it}}\notag\\
        & \geq (1-2\eps).\frac{\sum_{i\in \Nc}v_ix^*_{it}}{1+\sum_{i\in \Nc}v_ix^*_{it}}.\notag
    \end{align}
    In \black{Equation~\eqref{eq:13}}, we apply Claim \ref{cl:nonlightcust}. In \black{Equation~\eqref{eq:14}}, we use Equation \eqref{eq:boundbound}, as well as the monotonicity of $x\mapsto x/(1+x)$.
\end{itemize}

\noindent {\em {Case 3}: Customer $t$ is light.} \black{The following claim is simply a restatement of Lemma~\ref{lem:nearoptimal} when customer $t$ is light. Due to the complex nature of its proof, we defer it to Appendix~\ref{apx:lightcust}}.
\begin{claim}\label{cl:lightcusts}
    $
        \E\left[\frac{V(S_t)}{1+V(S_t)}\right] \geq (1-2\eps)\cdot \frac{\sum_{i\in \Nc}v_ix_{it}^*}{1+\sum_{i\in \Nc}v_ix_{it}^*}.
    $
\end{claim}

Combining cases 1, 2 and 3, it follows that\begin{equation*}
    \E\left[\sum_{t =1}^T\frac{V(S_t)}{1+V(S_t)}\right] \geq (1-2\eps)\cdot \sum_{t=1}^T \frac{\sum_{i\in \Nc}v_ix_{it}^*}{1+\sum_{i\in \Nc}v_ix_{it}^*}.\end{equation*}

\subsection{Proof of Claim \ref{cl:nonlightcust}}\label{apx:nonlightcust}
For the purpose of showing that $W^{\lar}_t = w_t^{\lar}$, note that \begin{align*}
    W_t^\lar& = \sum_{q = 1}^Q\sum_{i\in D_q}v_iX_{it}\\
    & = \sum_{q = 1}^Q \eps^5\cdot(1+\eps)^{q-1}\sum_{i\in D_q}X_{it}\\
    & = \sum_{q = 1}^Q \eps^5\cdot(1+\eps)^{q-1}\sum_{i\in D_q}x^*_{it}\\
    & = \sum_{q = 1}^Q\sum_{i\in D_q}v_ix^*_{it}\\
    & = w_t^\lar,
\end{align*}
where the second equality follows from the definition of $D_q$ as the set of products with weight $\eps\cdot (1+\eps)^{q-1}$, and the third inequality follows from the degree preservation property \ref{it:P2}, and the fact that $\sum_{i\in D_q}x^*_{it}$ is an integer, as stipulated by constraint (III) of \ref{eq:ILP}. Note that $k_{q,t}\neq \star$ for all $q =1,\ldots, Q$, as otherwise, customer $t$ would be heavy and unbounded, which contradicts the case hypothesis.

\subsection{Proof of Claim \ref{cl:lightcusts}}\label{apx:lightcust}
First, it is not difficult to verify that the proof of Claim \ref{cl:nonlightcust} also holds for light customers. Therefore, for every customer $t\in G_{\rm light}$, we can decompose $V(S_t)$ into the sum of : (1) The random variable $W_t^\sma$, and (2) The deterministic random variable $W_t^{\lar} \overset{a.s}{=}w_t^\lar$, i.e.,
$
    V(S_t) = W_t^\sma + w_t^\lar.    
$
Our proof consists on two steps. In the first step, we lower bound the expectation of $V(S_t)/(1+V(S_t))$, mainly using a convexity argument. Specifically, we prove that there exists a constant $\alpha\in (0,1]$ such that $
    \E\left[\frac{V(S_t)}{1+V(S_t)}\right] \geq \alpha\cdot \sum_{i\in \Nc}v_ix^*_{it}.
$
In the second step, we show that $\alpha = 1-O(\eps)$.
\paragraph{Step 1:}
Let $\cal W$ be the support of the random variable $W_t^\sma $, i.e., $
    {{\cal W} = \left\{w\in\R\,\colon\, \P\left[W_t^\sma= w\right]>0\right\}}.
$
Note that the support ${\cal W}$ is finite since the random variable $W_t^\sma$ is a deterministic function of the random binary vector $(X_{it}\,\colon\,i\in D_0)$, whose support is trivially finite.
We have
\begin{align*}
    \E\left[\frac{V(S_t)}{1+V(S_t)}\right] & = \E\left[\frac{W_t^\sma + w_t^\lar}{1+W_t^\sma + w_t^\lar}\right]\\
    & = \sum_{w \in {\cal W}} \frac{w+w_t^\lar}{1+w+ w_t^\lar}\P\left[W_t^\sma=w\right]\\
    & = \left(\sum_{w\in {\cal W}} \frac{1}{1+w + w_t^\lar}\cdot \frac{(w+ w_t^\lar)\cdot \P\left[W_t^\sma=w\right]}{w_t^\lar+\E\left[W_t^\sma\right]}\right)\cdot \left(w_t^\lar +\E\left[W_t^\sma\right]\right).
\end{align*}
For each $w\in {\cal W}$, let $f(w) = \frac{1}{1+w + w_t^\lar}$ and let $z_w = \frac{(w+ w_t^\lar)\cdot \P\left[W_t^\sma=w\right]}{ w_t^\lar+\E\left[W_t^\sma\right]}$. It is easy to verify that $\sum_{w\in {\cal W}} z_w=1$, and that $f$ is a convex function on $(0, +\infty)$. Therefore, using Jensen's inequality, we have\begin{align*}
    \E\left[\frac{V(S_t)}{1+V(S_t)}\right] &\geq f\left(\sum_{w\in {\cal W}} z_w \cdot w\right)\cdot \left(\E[W_t^\sma]+ w_t^\lar\right)= \underbrace{f\left(\sum_{w\in {\cal W}} z_w \cdot w\right)}_{\coloneqq \alpha}\cdot \E\left[V(S_t)\right],
\end{align*}
where the last equality holds since $\E[W_t^\sma]+ w_t^\lar = \E[V(S_t)]$.
\paragraph{Step 2.} Let us now show that $\alpha = 1-O(\eps)$. To this end, note that
\begin{align}
    \frac1\alpha& = {1+w_t^\lar + \frac{\sum_{w\in {\cal W}}(w+ w_t^\lar)\cdot \P[W_t^\sma=w]\cdot w}{w_t^\lar+\E[W_t^\sma]}} \notag\\
    & = {1+w_t^\lar+\frac{w_t^\lar\cdot\E[W_t^\sma] + \E[(W_t^\sma)^2]}{w_t^\lar+\E[W_t^\sma]}}\notag\\
    & = {1+ w_t^\lar+\E[W_t^\sma] + \frac{\Var[W_t^\sma]}{w_t^\lar+\E[W_t^\sma]}}.\label{eq:last}
\end{align}
In order to show that $\alpha = 1-O(\eps)$, we proceed by bounding the denominator in Equation \eqref{eq:last}. For this purpose, we have\begin{align}
    \Var[W_t^\sma] & = \Var\left[\sum_{i\in D_0} v_iX_{it}\right]\notag\\
    & = \sum_{i\in D_0}v_i^2\cdot \Var\left[X_{it}\right]+ \sum_{\substack{i,j\in D_0\\ i\neq j}} v_i\cdot v_j\cdot \Cov[X_{it},X_{jt}] \label{eq:420}\\
        & \leq \sum_{i\in D_0} v_i^2\cdot \Var[X_{it}]\notag\\
        &\leq \eps^5\cdot \sum_{i\in D_0} v_i\cdot \E[X_{it}] \label{eq:421}\\
        &= \eps^5\cdot \E[W_t^\sma]\notag.
\end{align}
In Equation \eqref{eq:421}, we use the fact that for any product $i\in D_0$, $\Var[X_{it}]\leq \E[X_{it}]$ since $X_{it}$ is a Bernoulli random variable, and the fact that $v_i \leq \eps^5$.
Equation \eqref{eq:420} holds since $\Cov[X_{it}, X_{jt}]\leq 0$ for all $i,j\in D_0$, $i\neq j$. Indeed, \begin{align*}
    \Cov[X_{it}, X_{jt}] &=\E[X_{it}\cdot X_{jt}] - \E[X_{it}]\cdot \E[X_{jt}]\\
    & = \P[(X_{it} = 1)\wedge (X_{it} = 1)]-\P[X_{it} = 1]\cdot \P[X_{it} = 1]\\
    &\leq 0,
\end{align*}
where the inequality holds due to the negative correlation property \ref{it:P3}.
Therefore, by replacing in Equation \eqref{eq:last}, we have\begin{align*}
    \alpha &\geq \frac{1}{1+w_t^\lar+\E[W_t^\sma]+\frac{\eps^5\cdot \E[W_t^\sma]}{w_t^\lar+\E[W_t^\sma]}} \\
    & \geq \frac{1}{1+\eps+\eps^5} \\
    &\geq \frac{1}{1+2\eps}\\
    &\geq 1-2\eps
\end{align*}
In the second inequality, we use the fact that $w_t^\lar+\E[W_t^\sma] = \E[V(S_t)] = \sum_{i\in \Nc}v_ix_{it}^* \leq \eps$, by constraint (VI) of LP, as customer $t$ is light by the case hypothesis.

\section{Proof of Claim~\ref{cl:increasefee}} \label{apx:increasefee}
Recall that an optimal solution to \ref{APV} is given by $S_t^* = \overline{L_t\cup \cdots \cup L_T}$ for every customer $t\in [T]$, as shown in Theorem \ref{Solution structure}. \black{Using Equations \eqref{eq:theells} and \eqref{eq:sol} to characterize the optimal solution of this updated instance, we define $\Tilde L_t = \{i\in\Nc:\Tilde \ell_i = t\}$, for all $t\in \{0\}\cup[T]$ and $\Tilde S_t = \llbracket\bigcup_{t\leq u\leq T}\Tilde L_u\rrbracket$ for all $t \in [T]$. In particular, 
    $$
    \Tilde L_t := \left\{\begin{array}{ll}
         L_t &  \quad\text{if }t\notin \{\ell_i, \ell_i+1\} \\
         L_t\setminus \{i\} & \quad\text{if } t=\ell_i\\
         L_t\cup\{i\} & \quad\text{if } t=\ell_i+1
    \end{array}\right.
    $$
    Therefore, we have $\Tilde S_t = S_t^*$ for all $t\neq \ell_i+1$, and $\Tilde S_{\ell_i+1} = \overline{L_t\cup\cdots\cup L_{\ell_i+1}\cup\{i\}}$.
}
Let $\Tilde \Delta$ be the total revenue loss incurred in the modified instance, that is $\Tilde \Delta = T\cdot R(S^*) - \sum_{t=1}^T R(\Tilde S_t)$. Additionally, for any product $j\in \Nc$, let $\Tilde C_{j}$ be its contribution after increasing $\ell_i$, i.e., $\Tilde C_{j} = \sum_{t=1}^T \mathbbm{1}(j \in \Tilde S_t) (p_j - R(\Tilde S_t)) \cdot v_j$. Finally, let $\Tilde \Gamma_j = (\Tilde C_{j}^- / \sum_{m\neq j}\Tilde C_{m}^-)\cdot  \Tilde \Delta$, namely the fee imposed on vendor $j$ in the modified instance. Using these notations, let us show that $\Tilde \Gamma_i\geq \Gamma_i$.
 First, if $i\in S_{\ell_i+1}^*$, then $S_{\ell_i+1}^* = \Tilde S_{\ell_i+1}$ and we have $\Tilde C_{i} = \cont{i}$. Consequently, $\Tilde \Gamma_i  = \Gamma_i$. Otherwise, $\cont{i}\geq\Tilde C_{i}$, and therefore $C_{i}^-\leq\Tilde C_{i}^-$. Furthermore, since $R(S^*_{\ell_i+1}) \geq R(\Tilde S_{\ell_i+1})$, we have $\Tilde C_{j} \geq C_{j}$ for every $j\neq i$. Finally, we have\begin{align}
    \Tilde \Gamma_{i} &= \frac{\Tilde C_i^-}{\Tilde C_{i}^- + \sum_{j\neq i}\Tilde C_{j}^-} \cdot \Tilde \Delta\notag\\
    & \geq \frac{\Tilde C_i^-}{\Tilde C_{i}^- + \sum_{j\neq i}C_{j}^-}\cdot  \Delta\notag\\
    &\geq \frac{C_i^-}{{C_{i}^-} + \sum_{j\neq i}{C_{j}^-}}\cdot  \Delta\label{eq:2049}\\
    & = \Gamma_i,\notag
\end{align}
where Inequality \eqref{eq:2049} follows from the fact that $x\mapsto x/{(c+x)}$ is nondecreasing on $[0,+\infty)$, for any $c>0$.

\end{document}